\RequirePackage{etoolbox}
\csdef{input@path}{{style/}{graphics/}}
\documentclass[aop,MSNbibl,dvips]{arximspdf}
\makeatletter
   \@ifpackageloaded{graphicx}{}{\usepackage{graphicx}}
\makeatother
\usepackage{url,breakurl}

\doi{10.1214/13-AOP889} 
\volume{43}
\issue{3}
\pubyear{2015}
\firstpage{1274}
\lastpage{1314}

\makeatletter

\newcommand{\rright}{\right}
\newcommand{\lleft}{\left}
\newcommand{\rrvert}{\vert}
\newcommand{\llvert}{\vert}
\newtheorem{prop}{Proposition}[section]
\newtheorem{lemma}[prop]{Lemma}
\newtheorem{cor}[prop]{Corollary}
\newtheorem{theorem}[prop]{Theorem}
\newproclaim{remark}[prop]{Remark}
\makeatother

\begin{document}
\begin{frontmatter}

\title{Stein's method and the rank distribution of random~matrices
over finite fields}
\runtitle{Stein's method and the rank distribution}

\begin{aug}
\author[A]{\fnms{Jason} \snm{Fulman}\ead[label=e1]{fulman@usc.edu}\thanksref{T1}}
\and
\author[A]{\fnms{Larry} \snm{Goldstein}\corref{}\ead[label=e2]{larry@math.usc.edu}\thanksref{T2}}
\runauthor{J. Fulman and L. Goldstein}
\affiliation{University of Southern California}
\address[A]{Department of Mathematics\\
University of Southern California\\
Los Angeles, California 90089\\
USA\\
\printead{e1}\\
\phantom{E-mail:\ }\printead*{e2}} 
\end{aug}
\thankstext{T1}{Supported by a Simons Foundation Fellowship and NSA
Grant H98230-13-1-0219.}
\thankstext{T2}{Supported by NSA Grant H98230-11-1-0162.}

\received{\smonth{2} \syear{2013}}
\revised{\smonth{9} \syear{2013}}

%
\begin{abstract}
With ${\mathcal Q}_{q,n}$ the distribution of $n$ minus the rank of a
matrix chosen uniformly from the collection of all $n \times(n+m)$ matrices
over the finite field $\mathbb{F}_q$ of size $q \ge2$, and ${\mathcal
Q}_q$ the distributional limit of ${\mathcal Q}_{q,n}$ as $n
\rightarrow\infty$, we apply Stein's method to prove the total
variation bound
\[
\frac{1}{8 q^{n+m+1}} \leq\|{\mathcal Q}_{q,n}-{\mathcal
Q}_q\|_{\mathrm{TV}} \leq\frac{3}{q^{n+m+1}}.
\]
In addition, we obtain similar sharp results for the rank distributions
of symmetric, symmetric with zero diagonal, skew symmetric, skew
centrosymmetric and Hermitian matrices.
\end{abstract}

%

\end{frontmatter}

\section{Introduction}\label{intro}
We study the distribution of the rank for various ensembles of random
matrices over finite fields. To give a flavor of our results, let $M_n$
be chosen uniformly from all $n \times(n+m)$ matrices over the finite
field $\mathbb{F}_q$ of size $q \ge2$. Letting $Q_{q,n}=n-\operatorname{rank}(M_n)$, it is known (page 38 of \cite{Be}) that for all $k$ in~$U_n=\{0,\ldots,n\}$,
%
\begin{eqnarray}\label{defQqnrect}
P(Q_{q,n}=k)=p_{k,n}
\nonumber\\[-8pt]\\[-8pt]
\eqntext{\displaystyle \mbox{where }
p_{k,n}=\frac{1}{q^{k(m+k)}} \frac{\prod_{i=1}^{n+m} (1-1/q^i) \prod
_{i=k+1}^n (1-1/q^i)}{\prod_{i=1}^{n-k} (1-1/q^i) \prod_{i=1}^{m+k}
(1-1/q^i)}.} 
\end{eqnarray}
Clearly, for any fixed $k \in\mathbb{N}_0$, the collection of
nonnegative integers,
%
\begin{equation}
\label{defQqrect} \lim_{n \rightarrow\infty}p_{k,n}=p_k
\qquad\mbox{where } p_k = \frac
{1}{q^{k(m+k)}}
\frac{\prod_{i=k+1}^{\infty} (1-1/q^i)}{\prod_{i=1}^{m+k} (1-1/q^i)}. %
\end{equation}
For readability and notational agreement with the examples that follow,
we suppress $m$ in the definition of these distributions. Throughout,
we also adopt the convention that an empty product takes the value 1.
One of our main results, Theorem~\ref{main}, provides sharp upper and
lower bounds on the total variation distance between ${\mathcal
Q}_{q,n}$, the distribution of $Q_{q,n}$ in (\ref{defQqnrect}) and
its limit in (\ref{defQqrect}), denoted ${\mathcal Q}_q$. Recall
that the total variation distance between two probability distributions
$P_1,P_2$ on a finite set $S$ is given by
%
\begin{equation}
\label{defTV} \|P_1-P_2\|_{\mathrm{TV}}:=
\frac{1}{2} \sum_{s \in S} \bigl|P_1(s)-P_2(s)\bigr|
= \max_{A \subset S} \bigl|P_1(A)-P_2(A)\bigr|.
\end{equation}

\begin{theorem} \label{main}
For $q \ge2, n \geq1$ and $m \ge0$,
%
\begin{equation}
\frac{1}{8 q^{n+m+1}} \leq\|{\mathcal Q}_{q,n}-{\mathcal
Q}_q\|_{\mathrm{TV}} \leq\frac{3}{q^{n+m+1}}.
\end{equation}
\end{theorem}

The upper bound in Theorem~\ref{main} appears quite difficult to
compute directly by substituting the expressions for the point
probabilities given in (\ref{defQqnrect}) and (\ref{defQqrect})
into the defining expressions for the total variation distance in (\ref
{defTV}). In particular, even when $m=0,n=2$, the $p_{k,n}$ are not
monotonic in $k$.
On the other hand, use of Stein's method \cite{Stn,CGS} makes
for a quite tractable computation. In Sections~\ref{symmetric}--\ref{Hermitian},
we also apply our methods to ensembles of random matrices with symmetry
constraints, in particular, to symmetric, symmetric with zero diagonal,
skew symmetric,
skew centrosymmetric and Hermitian matrices.

Next, we give five pointers to the large literature on the rank
distribution of random matrices over finite fields, demonstrating that
the subject is of interest. First, one of the earliest systematic
studies of ranks of random matrices from the finite classical groups is
due to
Rudvalis and Shinoda \cite{RS,Sh}. They determine the rank
distribution of random matrices from
finite classical groups, and relate distributions such as ${\mathcal
Q}_q$ of (\ref{defQqrect}) to identities of Euler. Second, ranks of
random matrices from finite classical groups appear in works on the
``Cohen--Lenstra heuristics'' of number theory;
see \cite{Wa} for the finite general linear groups and \cite{Mal} for
the finite symplectic groups. Third, the rank
distribution of random matrices over finite fields is useful in coding
theory; see \cite{BS} and Chapter~15 of \cite{MS}. Fourth,
the distribution of ranks of uniformly chosen random matrices over
finite fields has been used to test
random number generators \cite{DGM}, and there is interest in the rate
of convergence to ${\mathcal Q}_q$.
Fifth, there is work on ranks of random matrices over finite fields
where the matrix entries are independent and identically
distributed, but not necessarily uniform. For example, the paper \cite
{CRR} uses a combination of M\"{o}bius inversion,
finite Fourier transforms and Poisson summation, to find conditions on
the distribution of matrix entries under which
the probability of a matrix being invertible tends to $p_0$ as $n
\rightarrow\infty$. Further results in
this direction, including rank distributions of sparse matrices, can be
found in \cite{BKW,Co1,Co2,KK}. It would be
valuable (but challenging) to extend our methods to these settings.

The organization of this paper is as follows. Section~\ref{Stein}
provides some general tools for our application of Stein's method, and
useful bounds
on products such as $\prod_i (1-1/q^i)$. The development followed here
is along the lines of the ``comparison of generators'' method as in
\cite{GR} and \cite{H}. Section~\ref{uniform} treats the rank
distribution of uniformly chosen $n \times(n+m)$ matrices over a
finite field, proving Theorem~\ref{main}. Section~\ref{symmetric}
treats the rank distribution of random symmetric matrices over a finite
field. Section~\ref{0diag} provides results for the rank distribution
of a uniformly chosen symmetric matrix with 0 diagonal; these are
called ``symplectic'' matrices in Chapter~15 of \cite{MS}, which uses
their rank distribution in the context of error correcting codes. The
same formulas for the rank distribution of symmetric matrices with zero
diagonal also apply to the rank distribution of random skew-symmetric
matrices, when $q$ is odd. Section~\ref{skew} treats the rank
distribution of random skew centrosymmetric matrices over finite
fields, and Section~\ref{Hermitian} treats the rank distribution of
random Hermitian matrices over finite fields. The \hyperref[app]{Appendix} gives an
algebraic proof, for the special case $m=0$ of square matrices, of the
crucial fact (proved probabilistically in Section~\ref{uniform} in
general) that if $Q_n$ has distribution ${\mathcal Q}_{q,n}$ of (\ref
{defQqnrect}), then $E(q^{Q_n})= 2 - 1/q^n$.

In the interest of notational simplicity, in Sections~\ref{symmetric}--\ref{Hermitian}, the specific rank distributions of the $n
\times n$ matrices of interest, and their limits, will apply only
locally in the section or subsection that contains them, and will there
be consistently denoted by ${\mathcal Q}_{q,n}$ and ${\mathcal Q}_q$,
respectively.

\section{Preliminaries} \label{Stein}
We begin with a general result for obtaining characterizations of
discrete integer distributions. We note that a version of Lemma~\ref{abcharlem} can be obtained by replacing $f(x)$ by $f(x)b(x)$ in
Theorem 2.1 of \cite{LS}, followed by a reversal of the interval
$[a,b]$, with similar remarks applying to the use of Proposition~2.1
and Corollary 2.1 of \cite{GR}. However, the following lemma and its
short, simple proof contain the precise conditions used throughout this
work and keep the paper self-contained.

We say a nonempty subset $\mathbb{I}$ of the integers $\mathbb{Z}$ is
an interval if $a,b \in\mathbb{I}$ with $a \le b$ then $[a,b] \cap
\mathbb{Z} \subset\mathbb{I}$. Let ${\mathcal L}(X)$ denote the
distribution of a random variable $X$.

%
\begin{lemma} \label{abcharlem}
Let $\{r_k, k \in\mathbb{I}\}$ be the distribution of a random
variable $Y$ having support the integer interval $\mathbb{I}$. Then if
$a(k)$ and $b(k)$ are any functions such that
%
\begin{equation}
\label{akrk-1=bkrk} a(k)r_{k-1}=b(k)r_k\qquad\mbox{for all $k \in
\mathbb{Z}$,}
\end{equation}
then a random variable $X$ having distribution ${\mathcal L}(Y)$ satisfies
%
\begin{equation}
\label{abchar} E\bigl[a(X+1)f(X+1)\bigr]=E\bigl[b(X)f(X)\bigr]
\end{equation}
for all functions $f\dvtx \mathbb{Z} \rightarrow\mathbb{R}$ for which the
expectations in (\ref{abchar}) exist.

Conversely, if $a(k)$ and $b(k)$ satisfy (\ref{akrk-1=bkrk}) and $a(k)
\neq0$ for all $k \in\mathbb{I}$ then $X$ has distribution
${\mathcal L}(Y)$ whenever $X$ has support $\mathbb{I}$ and satisfies
(\ref{abchar}) for all functions $f(x)={\mathbf1}(x=k), k \in\mathbb{I}$.

When $Y$ has support $\mathbb{N}_0$ then $k \in\mathbb{Z}$ in (\ref
{akrk-1=bkrk}) may be replaced by $k \in\mathbb{N}_0$, while if $Y$
has support $U_n=\{0,1,\ldots,n\}$ for some $n \in\mathbb{N}_0$,
then (\ref{akrk-1=bkrk}) may be replaced by the condition that (\ref
{akrk-1=bkrk}) holds for $k \in U_n$ and that
$a(n+1)=0$.
\end{lemma}

\begin{pf} First suppose that (\ref{akrk-1=bkrk}) holds and that
${\mathcal L}(X)={\mathcal L}(Y)$. Then for all $k \in\mathbb{Z}$,
\begin{eqnarray*}
E\bigl(a(X+1){\mathbf1}(X+1=k)\bigr) & = & a(k)P(X=k-1)
\\
&=& a(k)r_{k-1}
\\
&=& b(k)r_k
\\
& = & b(k)P(X=k)
\\
& = & E\bigl(b(X){\mathbf1}(X=k)\bigr).
\end{eqnarray*}
Hence, (\ref{abchar}) holds for $f(x)={\mathbf1}(x=k), k \in\mathbb
{Z}$. By linearity, (\ref{abchar}) holds for all functions with
finite support, and hence for all the claimed functions by dominated
convergence.

Conversely, if (\ref{abchar}) holds for $X$ with $f(x)={\mathbf1}(x=k)$
for $k \in\mathbb{I}$ then
\[
a(k)P(X=k-1)=b(k)P(X=k).
\]
Hence, using that $a(k) \neq0, r_k \neq0$ for $k \in\mathbb{I}$
and that $X$ has the same support as $Y$ yields
\[
\frac{P(X=k-1)}{P(X=k)} = \frac{b(k)}{a(k)} = \frac{r_{k-1}}{r_k}.
\]
If $\mathbb{I}=\{s,\ldots,t\}$, then for $j \in\mathbb{I}$
\[
\frac{P(X=j)}{P(X=t)} = \prod_{k=j+1}^t
\frac
{P(X=k-1)}{P(X=k)}=\frac{r_j}{r_t}.
\]
Summing over $j \in\mathbb{I}$ yields $P(X=t)=r_t$, and hence
$P(X=j)=r_j$, showing ${\mathcal L}(X)={\mathcal L}(Y)$. One may argue
similarly for the remaining cases where $\mathbb{I}$ is an unbounded
integer interval.

Lastly, when the support of $Y$ is a subset of $\mathbb{N}_0$ then
(\ref{akrk-1=bkrk}) holds trivially for \mbox{$k \notin\mathbb{N}_0$},
and when $Y$ has support $U_n=\{0,1,\ldots,n\}$ then (\ref
{akrk-1=bkrk}) also holds trivially for $k \ge n+2$, and at $k=n+1$
when $a(n+1)=0$.
\end{pf}

For example, when $Y$ has the Poisson distribution ${\mathcal
P}(\lambda)$ with parameter $\lambda$, then $r_k=e^{-\lambda}\lambda
^k/k!$, and we obtain
\[
\frac{r_{k-1}}{r_k}= \frac{k}{\lambda} \qquad\mbox{for all $k \in
\mathbb{N}_0$.} 
\]
Setting $b(k)=k$ and $a(k)=\lambda$ yields the standard
characterization of the Poisson distribution \cite{bhj},
\[
E\bigl[\lambda f(Y+1)\bigr]=E\bigl[Yf(Y)\bigr].
\]

Of particular interest here is the characterization (\ref{abchar}) of
Lemma~\ref{abcharlem} for limiting distributions ${\mathcal Q}_q$
with distribution $P(Q=k)=p_k$ having support $\mathbb{N}_0$. In this
case, when applying Lemma~\ref{abcharlem} we take $a(k)>0$ for all
$k \in\mathbb{N}_0$, whence $b(0)=0$ by~(\ref{akrk-1=bkrk}), and let
the values of $a(k)$ and $b(k)$ for $k \notin\mathbb{N}_0$ be
arbitrary. For such functions $a(k)$ and $b(k)$, we consider solutions
$f$ to recursive ``Stein equations'' of the form
%
\begin{equation}
\label{receq} a(k+1)f(k+1)-b(k)f(k)=h(k)-{\mathcal Q}_qh \qquad\mbox{for $k \in\mathbb{N}_0$,}
\end{equation}
where ${\mathcal Q}_qh=Eh(Q)$.

Solving (\ref{receq}) for $f(k), k \in\mathbb{N}_0$, when the
functions $a(k),b(k)$ satisfy only \mbox{$b(0)=0$} and $a(k)>0$ one may take
$f(0)=0$ arbitrarily, and easily verify that the remaining values are
uniquely determined and given by
%
\begin{equation}
\label{ratioabsol} f(k+1) = \sum_{j=0}^k
\biggl( \frac{\prod_{l=j+1}^k b(l) }{\prod_{l=j+1}^{k+1} a(l)} \biggr)
\bigl[h(j)-{\mathcal Q}_qh\bigr]
\qquad\mbox{for $k \in\mathbb{N}_0$.}
\end{equation}
In the case where the distribution $\{p_k, k \in\mathbb{N}_0\}$ with
support $\mathbb{N}_0$ satisfies (\ref{akrk-1=bkrk}) with~$p_k$
replacing $r_k$, the solution (\ref{ratioabsol}) simplifies to
%
\begin{eqnarray}
\label{steinsol} f(k+1) & = & \frac{1}{a(k+1)p_k}\sum_{j=0}^k
\bigl[h(j)-{\mathcal Q}_q h\bigr]p_j
\\[-2pt]
& = & \frac{E[(h(Q)-{\mathcal Q}_q h){\mathbf1}(Q \le k)]}{a(k+1) p_k}\qquad\mbox{for $k \in\mathbb{N}_0$.}
\end{eqnarray}

In particular, for $h_A(k) = {\mathbf1}(k \in A)$ with $A \subset\mathbb
{N}_0$ and $U_k=\{0,1,\ldots,k\}$, as in Barbour et al. \cite{bhj},
Lemma 1.1.1, for $k \in\mathbb{N}_0$, as
${\mathcal Q}_q h_A=P(Q \in A)$, the numerator of (\ref{steinsol}) is
given by
\[
P(Q \in A \cap U_k) - P(Q \in A)P(Q \in U_k).
\]
Now replacing $P(Q \in A \cap U_k)$ and $P(Q \in A)$ in the first and
second term, respectively, by
\[
P(Q \in A \cap U_k) \bigl[P(Q \in U_k) + P\bigl(Q \in
U_k^c\bigr) \bigr]
\]
and
\[
P(Q \in A \cap U_k) + P\bigl(Q \in A \cap U_k^c
\bigr),
\]
canceling the resulting common factor demonstrates that the solution
$f_A$ satisfies
%
\begin{eqnarray}
&& f_A(k+1) \nonumber
\\[-2pt]
&&\qquad = \frac{P(Q \in A \cap U_k)P(Q \in U_k^c)-P(Q \in A \cap U_k^c)P(Q
\in U_k)}{a(k+1) p_k} \label{forbound}
\\[-2pt]
&&\qquad \le \frac{P(Q \in A \cap U_k)P(Q \in U_k^c)}{a(k+1) p_k}\nonumber
\\[-2pt]
&&\qquad \le \frac{P(Q \in U_k)P(Q \in U_k^c)}{a(k+1) p_k}\label{prodUUcbound}
\end{eqnarray}
with equality when $A=U_k$. Since $f_{A^c}(k) = -f_A(k)$ the bound (\ref{prodUUcbound}) holds for $|f_A(k+1)|$.

\begin{lemma} \label{fA1boundlem}
Let $Q$ have distribution $\{p_k, k \in\mathbb{N}_0\}$ with $p_k>0$
for all $k \in\mathbb{N}_0$, and let $a(k),b(k)$ satisfy (\ref
{akrk-1=bkrk}) with $p_k$ replacing $r_k$, and for $A \subset\mathbb
{N}_0$ let
$f_A$ be the solution to (\ref{receq}) given by (\ref{forbound}). Then
\[
\bigl|f_A(1)\bigr| \le\frac{P(Q \ge1)}{a(1)}.
\]
\end{lemma}

\begin{pf} From (\ref{forbound}) with $k=0$, we obtain
\begin{eqnarray*}
f_A(1) &=& \frac{P(Q \in A \cap U_0)P(Q \ge1)-P(Q \in A \cap
U_0^c)P(Q=0)}{a(1) p_0}.
\end{eqnarray*}
If $A \ni0$, then
\begin{eqnarray*}
\bigl|f_A(1)\bigr| &=& \biggl\llvert\frac{P(Q=0)P(Q \ge1)-P(Q \in A \setminus\{
0\}
)P(Q=0)}{a(1) p_0}\biggr\rrvert
\\
&=& \frac{P(Q \ge1)-P(Q \in A \setminus\{0\})}{a(1)} \le\frac{P(Q
\ge1)}{a(1)},
\end{eqnarray*}
while if $A \not\ni0$ then again
\begin{eqnarray*}
\bigl|f_A(1)\bigr| &=& \frac{P(Q \in A)P(Q=0)}{a(1) p_0} \le\frac{P(Q \ge1)}{a(1)}.
\end{eqnarray*}\upqed
\end{pf}

Lemma~\ref{bound} collects some bounds that will be useful. We first
state the simple inequality
%
\begin{equation}
\label{prodagesuma} \prod_{i=1}^n
(1-a_i) \ge1 -\sum_{i=1}^n
a_i
\end{equation}
valid for $a_i \in[0,1], i=1,\ldots,n$, and easily shown by induction.

\begin{lemma} \label{bound} Let $q \geq2$. Then
\begin{eqnarray*}
\prod_{i=1}^n\bigl(1-1/q^i
\bigr) &\geq& 1-1/q-1/q^2,
\\
\prod_{i \geq1} \bigl(1-1/q^i\bigr)
&\geq& 1-1/q-1/q^2+1/q^5+1/q^7-1/q^{12}-1/q^{15},
\\
\mathop{\prod_{i \geq1}}_{i\ \mathit{odd}} \bigl(1-1/q^i\bigr)
&\geq& 1-1/q-1/q^3
\end{eqnarray*}
and
\[
\mathop{\prod_{i \geq3}}_{i\ \mathit{odd}} \bigl(1-1/q^i\bigr) \geq1 -2/q^3.
\]

For $0 \le m+1 \le n$,
\[
\prod_{i=m+1}^n\bigl(1-1/q^i
\bigr) \geq1-2/q^{m+1}.
\]
\end{lemma}

\begin{pf} The first claim is Lemma 3.5 of \cite{NP}, and arguing
as there yields the second claim. Thus,
\begin{eqnarray*}
\mathop{\prod_{i \geq1}}_{i\; \mathrm{odd}} \bigl(1-1/q^i\bigr) & \geq&
\frac{\prod_{i
\geq1} (1-1/q^i)}{1-1/q^2}
\\[-2pt]
& \geq& \frac{1-1/q-1/q^2+1/q^5+1/q^7-1/q^{12}-1/q^{15}}{1-1/q^2}
\\[-1pt]
& \geq& 1-1/q-1/q^3,
\end{eqnarray*}
where the last inequality holds since
\begin{eqnarray*}
&& \bigl( 1-1/q-1/q^2+1/q^5+1/q^7-1/q^{12}-1/q^{15}
\bigr)
\\[-1pt]
&&\qquad{} - \bigl(1-1/q^2\bigr) \bigl(1-1/q-1/q^3\bigr)
 =\frac{q^8-q^3-1}{q^{15}},
\end{eqnarray*}
%
which is positive for $q \geq2$. The next inequality now follows by
applying the one just shown to obtain
\[
\mathop{\prod_{i \geq3}}_{i\; \mathrm{odd}} \bigl(1-1/q^i\bigr) \geq1-
\frac{1}{q^3(1-1/q)}
\]
and using that $q \geq2$.

For the final claim, using (\ref{prodagesuma}) yields
\begin{eqnarray*}
\prod_{i=m+1}^n \bigl(1 -
1/q^i \bigr) &\ge& 1-\sum_{i=m+1}^n
1/q^i
\\[-1pt]
&\ge& 1-\sum_{i=m+1}^\infty1/q^i
\\[-1pt]
&=& 1-\frac{1}{q^{m+1}(1-1/q)}
\\[-1pt]
&\ge& 1-\frac{2}{q^{m+1}}.
\end{eqnarray*}\upqed
\end{pf}

\begin{remark*}
Since
\[
\prod_{i=1}^n \bigl(1-1/q^i
\bigr) \geq\prod_{i \geq1} \bigl(1-1/q^i
\bigr),
\]
it is easy to see that the
second claim of Lemma~\ref{bound} implies the first.
\end{remark*}

\section{Uniform matrices over finite fields} \label{uniform}

In this section, we study the rank distribution of matrices chosen
uniformly from those of dimension $n \times(n+m)$ with entries from
the finite field $\mathbb{F}_q$, and take the distributions ${\mathcal
Q}_q$ and ${\mathcal Q}_{q,n}$ as in~(\ref{defQqrect}) and (\ref
{defQqnrect}), respectively; throughout this section, we take $q \ge
2$. The goal of this section is to prove Theorem~\ref{main}.

The following lemma is our first application of the characterizations
provided by Lemma~\ref{abcharlem}.

\begin{lemma} 
If $Q$ has the ${\mathcal Q}_q$ distribution then
%
\begin{equation}
\label{Qqchar} E \bigl[qf(Q+1) \bigr]=E \bigl[\bigl(q^Q-1\bigr)
\bigl(q^{Q+m}-1\bigr) f(Q) \bigr]
\end{equation}
for all functions $f$ for which these expectations exist.

If $Q_n$ has the ${\mathcal Q}_{q,n}$ distribution, then
%
\begin{equation}
\label{Qnqchar} E \bigl[q\bigl(1-q^{-n+Q_n}\bigr)f(Q_n+1)
\bigr]=E \bigl[\bigl(q^{Q_n}-1\bigr) \bigl(q^{Q_n+m}-1\bigr)
f(Q_n) \bigr]
\end{equation}
for all functions $f$ for which these expectations exist.
\end{lemma}

\begin{pf} From (\ref{defQqrect}), we obtain
\[
\frac{p_{k-1}}{p_k} = \frac{(q^k-1)(q^{m+k}-1)}{q} \qquad\mbox{for all
$k \in
\mathbb{N}_0$.}
\]
An application of Lemma~\ref{abcharlem} with $a(k)=q$ and
$b(k)=(q^k-1)(q^{k+m}-1)$ yields~(\ref{Qqchar}). Similarly, from
(\ref{defQqnrect}) we obtain
%
\begin{equation}
\label{defQqnrect17} \frac{p_{k-1,n}}{p_{k,n}}=\frac
{(q^k-1)(q^{k+m}-1)}{q(1-q^{-n+k-1})} \qquad\mbox{for all $k \in
U_n$.}
\end{equation}
An application of Lemma~\ref{abcharlem} with $a(k)=q(1-q^{-n+k-1})$,
$b(k)=(q^k-1)(q^{k+m}-1)$, noting $a(n+1)=0$, yields (\ref{Qnqchar}).
\end{pf}

Here, we calculate $E(q^{Q_n})$ using the characterization (\ref
{Qnqchar}). An algebraic proof for the case $m=0$ of Lemma~\ref{qtoX}
appears in the \hyperref[app]{Appendix}. After reading the first version of this paper,
Dennis Stanton has shown us a proof of this special case using the
$q$-Chu--Vandermonde summation formula.

\begin{lemma} \label{qtoX} If $Q_n$ has the ${\mathcal Q}_{q,n}$
distribution on $U_n=\{0,1,\ldots,n \}$ given by~(\ref
{defQqnrect}), then
\[
E\bigl(q^{Q_n}\bigr) = 1 + q^{-m} - q^{-(n+m)}.
\]
\end{lemma}

\begin{pf} Applying the characterization (\ref{Qnqchar}) with the
choice $f(x)=q^{kx}$, we obtain
\[
E\bigl[q\bigl(1-q^{-n+Q_n}\bigr)q^{k(Q_n+1)}\bigr] = E\bigl[
\bigl(q^{Q_n}-1\bigr) \bigl(q^{Q_n+m}-1\bigr) q^{kQ_n}
\bigr].
\]
Letting $c_k=Eq^{kQ_n}$ yields the recursion
%
\begin{equation}
\label{recrect} q^m c_{k+2} = \bigl(1+q^m-q^{-n+k+1}
\bigr)c_{k+1} + \bigl(q^{k+1}-1\bigr) c_k.
\end{equation}
Since ${\mathcal Q}_{q,n}$ is a probability distribution, $c_0=1$, and
setting $k=-1$ in (\ref{recrect}) yields the claim.
\end{pf}

In the remainder of this section, we consider the Stein equation (\ref
{receq}), with
%
\begin{equation}
\label{defak-bkuniform} a(k)=q \quad\mbox{and} \quad b(k)=\bigl
(q^k-1\bigr)
\bigl(q^{k+m}-1\bigr)
\end{equation}
for the target distribution ${\mathcal Q}_q$, and for $A \subset
\mathbb{N}_0$ we let $f_A$ denote the solution (\ref{steinsol}) when
$h(k)={\mathbf1}(k \in A)$.

For a function $f\dvtx \mathbb{N}_0 \rightarrow\mathbb{R}$, let
\[
\|f\| = \sup_{k \in\mathbb{N}_0} \bigl|f(k)\bigr|.
\]

\begin{lemma} \label{fbound}
The solution $f_A$ satisfies
%
\[
\sup_{A \subset\mathbb{N}_0}\|f_A\| \le\frac{2}{q^{m+2}}.
\]
If $m=0$, the bound can be improved to
\[
\sup_{A \subset\mathbb{N}_0}\|f_A\| \le\frac{1}{q^{2}} +
\frac{1}{q^{3}}.
\]
\end{lemma}

\begin{pf} As we may set $f_A(0)=0$, it suffices to consider $f_A(k+1)$ for
$k \in\mathbb{N}_0$.
By Lemma~\ref{fA1boundlem}, for all $A \subset\mathbb{N}_0$
\begin{eqnarray*}
\bigl|f_A(1)\bigr| & \le& \frac{P(Q \ge1)}{q}
\\
& = & \frac{1-p_0}{q}
\\
& = & \frac{1}{q} \biggl( 1 -\prod_{i \ge m+1}
\biggl(1-\frac{1}{q^i}\biggr) \biggr)
\\
& \le& \frac{2}{q^{m+2}},
\end{eqnarray*}
where we have applied the last part of Lemma~\ref{bound}. For $m=0$,
using the first inequality of Lemma~\ref{bound} in
the last step gives that
\[
\bigl|f_A(1)\bigr| \leq\frac{1}{q^2} + \frac{1}{q^3}.
\]

Now consider the case $k \ge1$. By (\ref{prodUUcbound}) and (\ref
{defak-bkuniform}), we have
%
\begin{equation}
\label{prodU2} \bigl|f_A(k+1)\bigr| \le\frac{P(Q \in U_k)P(Q \in U_k^c)}{q p_k}
\end{equation}
and by neglecting the term $P(Q \in U_k)$ in (\ref{prodU2}) and
applying (\ref{defQqrect}) we obtain
\begin{eqnarray*}
&& \bigl|f_A(k+1)\bigr|
\\
&&\qquad \le \frac{P(Q \in U_k^c)}{q p_k}
\\
&&\qquad = q^{k(m+k)-1} \frac{\prod_{i=1}^{m+k} (1-1/q^i)}{\prod
_{i=k+1}^{\infty} (1-1/q^i)}\sum_{l=k+1}^\infty
\frac{1}{q^{l(m+l)}} \frac{\prod_{i=l+1}^{\infty} (1-1/q^i)}{\prod_{i=1}^{m+l}
(1-1/q^i)}
\\
&&\qquad = \frac{q^{k(m+k)-1}}{\prod_{i=k+1}^{\infty} (1-1/q^i)}\sum
_{l=k+1}^\infty
\frac{1}{q^{l(m+l)}} \frac{\prod_{i=l+1}^{\infty}
(1-1/q^i)}{\prod_{i=m+k+1}^{m+l} (1-1/q^i)}
\\
&&\qquad \le \frac{q^{k(m+k)-1}}{\prod_{i=k+1}^\infty(1-1/q^i) \prod
_{i=m+k+1}^\infty(1-1/q^i)} \sum_{l=k+1}^\infty
\frac{1}{q^{l(m+l)}}
\\
&&\qquad\le \frac{q^{k(m+k)-1}}{(1-\sum_{j=k+1}^\infty(1/q^j))
(1-\sum_{j=m+k+1}^\infty(1/q^j))}\sum_{l=k+1}^\infty
\frac{1}{q^{l(m+l)}}
\\
&&\qquad= \frac{q^{k(m+k)-1}}{(1-(q^{-(k+1)}/(1-q^{-1})))
(1-({q^{-(m+k+1)}}/({1-q^{-1}})))}\sum_{l=k+1}^\infty
\frac{1}{q^{l(m+l)}}
\\
&&\qquad= \frac{1}{q(1-(1/{q^k(q-1)}))(1-(1/{q^{m+k}(q-1)}))}\sum
_{l=1}^\infty
\frac{1}{q^{2lk+l^2+lm}}
\\
&&\qquad\le \frac{1}{q(q^k-(1/(q-1)))(q^{m+k}-(1/(q-1)))}\sum
_{l=1}^\infty\frac{1}{q^{l^2}},
\end{eqnarray*}
where for the third inequality we have applied (\ref{prodagesuma}).

We claim that
\[
4 \biggl( q^k-\frac{1}{q-1} \biggr) \biggl( q^{m+k}-
\frac{1}{q-1} \biggr) \ge q^{m+2}.
\]
As the left-hand side is increasing in $k \ge1$, it suffices to prove
the claim for $k=1$. In this case, the claim may be rewritten as
\[
3q^{m+2} + \frac{4}{(q-1)^2} \ge\frac{4}{q-1} \bigl(q +
q^{m+1} \bigr).
\]
As $q \ge2$, the result is a consequence of the two easily verified
inequalities
\[
2q^{m+2} \ge\frac{4q^{m+1}}{q-1} \quad\mbox{and} \quad q^{m+2} +
\frac
{4}{(q-1)^2} \ge\frac{4q}{q-1}.
\]

Hence, for $k \ge1$, using $q \ge2$, we obtain
\[
\bigl|f_A(k+1)\bigr| \le\frac{4}{q^{m+3}}\sum_{l=1}^\infty
\frac{1}{q^{l^2}} \le\frac
{4}{q^{m+3}} \Biggl( \frac{1}{2} + \sum
_{l=2}^\infty\frac
{1}{2^{2+l}} \Biggr) \le
\frac{1}{q^{m+2}}+\frac{1}{q^{m+3}},
\]
where the final inequality used that $2/q^{m+3} \leq1/q^{m+2}$, and
that $\sum_{l=2}^{\infty} \frac{1}{2^{2+l}} \leq1/4$, thus
completing the proof of the lemma.
\end{pf}

We now present the proof of Theorem~\ref{main}.

\begin{pf*}{Proof of Theorem~\ref{main}}
We first compute the lower bound on the total variation distance
by estimating the difference of the two distributions at $k=0$. In
particular, by (\ref{defTV}), (\ref{defQqnrect}) and (\ref{defQqrect}),
\begin{eqnarray*}
&& \|{\mathcal Q}_{q,n}-{\mathcal Q}_q\|_{\mathrm{TV}}
\\
&&\qquad  \geq \frac{1}{2} [ p_{0,n}-p_0]
\\
&&\qquad  =  \frac{1}{2} \biggl[ \prod_{m+1 \le i \le m+ n}
\bigl(1-1/q^i\bigr) - \prod_{i \geq m+1}
\bigl(1-1/q^i\bigr) \biggr]
\\
&&\qquad  \geq \frac{1}{2} \bigl[ \bigl(1-1/q^{m+1}\bigr) \cdots
\bigl(1-1/q^{n+m}\bigr) - \bigl(1-1/q^{m+1}\bigr) \cdots
\bigl(1-1/q^{n+m+1}\bigr) \bigr]
\\
&&\qquad  =  \frac{1}{2q^{n+m+1}} \bigl(1-1/q^{m+1}\bigr) \cdots
\bigl(1-1/q^{n+m}\bigr)
\\
&&\qquad \ge \frac{1}{2q^{n+m+1}} (1-1/q) \cdots\bigl(1-1/q^n\bigr)
\\
&&\qquad \geq \frac{1}{2q^{n+m+1}} \bigl(1-1/q-1/q^2\bigr)
\\
&&\qquad \geq \frac{1}{8q^{n+m+1}}.
\end{eqnarray*}
The fourth inequality used Lemma~\ref{bound}, and the last that $q \ge2$.

For the upper bound, with $h_A(k)={\mathbf1}(k \in A)$ we obtain
\begin{eqnarray*}
&& \bigl|P(Q_n \in A)-P(Q \in A)\bigr|
\\
&&\qquad  = \bigl|E\bigl[h_A(Q_n) \bigr]-{\mathcal Q}_qh_A\bigr|
\\
&&\qquad = \bigl|E\bigl[qf_A(Q_n+1) - \bigl(q^{Q_n}-1\bigr)
\bigl(q^{Q_n+m}-1\bigr) f_A(Q_n)\bigr]\bigr|
\\
&&\qquad =\bigl|E\bigl[q^{-n+Q_n+1}f_A(Q_n+1)\bigr]\bigr|
\le\|f_A\| Eq^{-n+Q_n+1},
\end{eqnarray*}
where we have applied (\ref{Qnqchar}) in the third equality. Applying
Lemmas~\ref{fbound} and~\ref{qtoX} gives that for $m \geq1$,
\begin{eqnarray*}
\|f_A\| Eq^{-n+Q_n+1} &\leq&\frac{2}{q^{m+2}} q^{-n+1}
\biggl(1+q^{-m}-\frac{1}{q^{n+m}} \biggr)
\\
&\le&\frac{2 (1+1/q )}{q^{n+m+1}}\le\frac{3}{q^{n+m+1}}.
\end{eqnarray*}
For $m=0$, applying Lemmas~\ref{fbound} and~\ref{qtoX} gives that
\begin{eqnarray*}
\|f_A\| Eq^{-n+Q_n+1} &\leq&\biggl( \frac{1}{q^2} +
\frac{1}{q^3} \biggr) q^{-n+1} \biggl(2-\frac{1}{q^{n}} \biggr)
\\
&\le& \frac{2 (1+1/q)}{q^{n+1}}\le\frac{3}{q^{n+1}}.
\end{eqnarray*}
Now taking the supremum over all $A \subset\mathbb{N}_0$ and applying
definition (\ref{defTV}) completes the proof.
\end{pf*}

\begin{remark*}
When $m=0$, the limit distribution ${\mathcal Q}_q$ also
arises in the study of the dimension of the fixed space of a random
element of $\operatorname{GL}(n,q)$. More precisely, Rudvalis and Shinoda \cite{RS}
prove that for $k$ fixed, as $n \rightarrow\infty$ the probability
that a random element of $\operatorname{GL}(n,q)$ has a $k$ dimensional fixed space
tends to $p_k$. See \cite{F3} for another proof.
\end{remark*}

\section{Symmetric matrices over finite fields} \label{symmetric}
Let $S$ be the set of symmetric matrices with entries in the finite
field $\mathbb{F}_q$ (where $q$ is
a prime power).\vspace*{-1pt} Clearly, $|S|=q^{{n+1 \choose2}}$. The paper \cite{C}
determines the rank distribution
of a matrix chosen uniformly from $S$ when $q$ is odd, and the paper
\cite{M} determines this distribution for $q$ both odd and even, given
by (\ref{pknzddef}).

Throughout this section $q \ge2$, and we let ${\mathcal Q}_q$ be the
distribution on $\mathbb{N}_0$ with mass function
%
\begin{equation}
\label{pkzddef} p_k = \frac{\prod_{i \geq1,\;i\; \mathrm{odd}} (1-1/q^i)}{
\prod_{i=1}^k
(q^i-1)}
\end{equation}
and for $n \in\mathbb{N}_0$ we let ${\mathcal Q}_{q,n}$ be the
distribution on $U_n=\{0,\ldots,n\}$ with mass function
%
\begin{eqnarray}\label{pknzddef}
p_{k,n} &=& \frac{N(n,n-k)}{q^{{n+1 \choose2}}}
\nonumber\\[16pt]\\[-30pt]
\eqntext{\displaystyle\mbox{where } N(n,2h) = \prod_{i=1}^h
\frac{q^{2i}}{(q^{2i}-1)} \prod_{i=0}^{2h-1}
\bigl(q^{n-i}-1\bigr)\mbox{ for $2h \le n$ and\hspace*{15pt}}}
\\
\eqntext{\displaystyle N(n,2h+1) = \prod_{i=1}^h
\frac{q^{2i}}{(q^{2i}-1)} \prod_{i=0}^{2h}
\bigl(q^{n-i}-1\bigr)\mbox{ for $2h+1\le n$.}}
\end{eqnarray}

\begin{theorem} \label{symmetricnonzerodiagthm}
If $n$ is even, we have
\[
\frac{0.18}{q^{n+1}} \leq\|{\mathcal Q}_{q,n}-{\mathcal
Q}_q\|_{\mathrm{TV}} \leq\frac{2.25}{q^{n+1}}.
\]
If $n$ is odd, we have
\[
\frac{0.18}{q^{n+2}} \leq\|{\mathcal Q}_{q,n}-{\mathcal
Q}_q\|_{\mathrm{TV}} \leq\frac{2}{q^{n+2}}.
\]
\end{theorem}

We again begin by using Lemma~\ref{abcharlem} to develop
characterizations for the two distributions of interest.
For $n \in\mathbb{N}_0$ we let ${\mathbf1}_n={\mathbf1}(\mbox{$n$ is
even})$, the indicator function that $n$ is even.

\begin{lemma} 
If $Q$ has the ${\mathcal Q}_q$ distribution then
\[
E\bigl[f(Q+1)\bigr]=E\bigl[\bigl(q^Q-1\bigr)f(Q)\bigr]
\]
for all functions $f$ for which these expectations exist.

If $Q_n$ has the ${\mathcal Q}_{q,n}$ distribution then
%
\begin{equation}
\label{steineqQnsymmetricnzd} E\bigl[\bigl(1-{\mathbf1}_{n-Q_n}q^{-(n-Q_n)}
\bigr)f(Q_n+1)\bigr]=E\bigl[\bigl(q^{Q_n}-1
\bigr)f(Q_n)\bigr]
\end{equation}
for all functions $f$ for which these expectations exist.
\end{lemma}

\begin{pf} By taking ratios in (\ref{pkzddef}), we obtain
\[
\frac{p_{k-1}}{p_k}=q^k-1.
\]
Setting $a(k)=1$ and $b(k)=q^k-1$ applying Lemma~\ref{abcharlem}
yields the first result.

If $n$ and $k$ are of the same parity, then $n-k=2h$ for some $h$, and
we have
\begin{eqnarray*}
\frac{p_{k-1,n}}{p_{k,n}} &=& \frac{N(n,n-k+1)}{N(n,n-k)}= \frac
{N(n,2h+1)}{N(n,2h)}
\\
&=& q^{n-2h}-1
= q^k-1, \qquad\mbox{$k \in U_n$.}
\end{eqnarray*}
In this case, we set $a(k)=1$ and $b(k)=q^k-1$.

If $k$ and $n$ are of opposite parity, then $n-k=2h+1$ for some $h$ and
we obtain
\begin{eqnarray*}
\frac{p_{k-1,n}}{p_{k,n}} &=& \frac{N(n,n-k+1)}{N(n,n-k)} = \frac
{N(n,2(h+1))}{N(n,2h+1)}
\\
&=&\frac{q^{2(h+1)}}{q^{2(h+1)}-1}\bigl(q^{n-2h-1}-1\bigr)
\\
&=& \frac
{q^{n-k+1}}{q^{n-k+1}-1}\bigl(q^k-1\bigr) =\frac{q^k-1}{1-q^{-n+k-1}}
\qquad\mbox{for $k \in U_n$.}
\end{eqnarray*}
In this case, we set $a(k)=1-q^{-n+k-1}$ and $b(k)=q^k-1$.

Writing $a(k)=1-{\mathbf1}_{n-k+1} q^{-n+k-1}$ and $b(k)=q^k-1$ combines
both cases. Noting that $a(n+1)=0$ an application of Lemma~\ref{abcharlem} completes the proof.
\end{pf}

\begin{lemma} \label{mom-symmetricnzd-even}
If $Q_n$ has distribution ${\mathcal Q}_{q,n}$ then
\[
E{\mathbf1}_{n-Q_n}q^{Q_n} = 1.
\]
\end{lemma}

\begin{pf} Setting $f(x)={\mathbf1}_{n-x}$ in (\ref
{steineqQnsymmetricnzd}) yields
\[
E\bigl[\bigl(1-{\mathbf1}_{n-Q_n}q^{-(n-Q_n)}\bigr){
\mathbf1}_{n-Q_n-1}\bigr]=E\bigl[\bigl(q^{Q_n}-1\bigr){\mathbf
1}_{n-Q_n}\bigr].
\]
Since ${\mathbf1}_{n-Q_n}{\mathbf1}_{n-Q_n-1}=0$, we obtain
\[
E[{\mathbf1}_{n-Q_n-1}]=E\bigl[\bigl(q^{Q_n}-1\bigr){
\mathbf1}_{n-Q_n}\bigr]
\]
and rearranging yields
\[
E\bigl[{\mathbf1}_{n-Q_n}q^{Q_n}\bigr]=E[{\mathbf1}_{n-Q_n-1}]+E[{
\mathbf1}_{n-Q_n}]=1
\]
as claimed.
\end{pf}

In the remainder of this section, we consider the Stein equation (\ref
{receq}) for the target distribution ${\mathcal Q}_q$ with
\[
a(k)=1 \quad\mbox{and} \quad b(k)=q^k-1
\]
and for $A \subset\mathbb{N}_0$ we let $f_A$ denote the solution
(\ref{steinsol}) when $h(k)={\mathbf1}(k \in A)$.

\begin{lemma} \label{solnzdsymmetric}
The solution $f_A$ satisfies
\[
\sup_{A \subset\mathbb{N}_0}\bigl|f_A(1)\bigr| \le\frac{1}{q}+
\frac{1}{q^3} \quad\mbox{and} \quad\sup_{A \subset\mathbb{N}_0, k \ge2}\bigl|f_A(k)\bigr|
\le\frac{2}{q^2}.
\]
\end{lemma}

\begin{pf} By Lemma~\ref{fA1boundlem}, for all $A \subset\mathbb{N}_0$,
\begin{eqnarray*}
\bigl|f_A(1)\bigr| & \le& P(Q \ge1)
\\
& = & 1-p_0
\\
& = & 1 -\prod_{i \ge1, \;i\; \mathrm{odd}}\biggl(1-\frac{1}{q^i}
\biggr)
\\
& \le& \frac{1}{q}+\frac{1}{q^3},
\end{eqnarray*}
where we applied the third inequality in Lemma~\ref{bound}.

For $k \ge1$, using (\ref{prodUUcbound}) and (\ref{pkzddef}),
\begin{eqnarray*}
\bigl|f_A(k+1)\bigr| &\le& \frac{P(Q \in U_k^c)}{p_k}
\\
&=& \prod_{i=1}^k
\bigl(q^i-1\bigr) \sum_{l=k+1}^\infty
\frac{1}{\prod_{i=1}^l(q^i-1)}
\\
&=& \sum_{l=k+1}^\infty\frac{1}{\prod_{i=k+1}^l(q^i-1)}
\\
&=& \sum_{l=k+1}^\infty\frac{1}{q^{(l(l+1)-k(k+1))/2}\prod_{i=k+1}^l(1-q^{-i})}
\\
&\le& \frac{1}{\prod_{i=k+1}^\infty(1-q^{-i})} \sum_{l=k+1}^\infty
\frac{1}{q^{(l(l+1)-k(k+1))/2}}
\\
&=&\frac{1}{\prod_{i=k+1}^\infty(1-q^{-i})} \Biggl( \frac
{1}{q^{k+1}} + \sum
_{l=2}^\infty\frac{1}{q^{lk+l^2/2+l/2}} \Biggr)
\\
&\le& \frac{1}{\prod_{i=k+1}^\infty(1-q^{-i})} \Biggl( \frac
{1}{q^{k+1}} + \frac{1}{q^{2k}}\sum
_{l=2}^\infty\frac
{1}{q^{l(l+1)/2}} \Biggr).
\end{eqnarray*}
In particular, for all $k \ge1$ we obtain
\[
\bigl|f_A(k+1)\bigr| \le\frac{1}{q^2\prod_{i=2}^\infty(1-q^{-i})} \Biggl(1+ \sum
_{l=2}^\infty\frac{1}{q^{l(l+1)/2}} \Biggr)
\]
and the proof is now completed by using the fact that for all $q \ge2$
\[
\frac{1}{\prod_{i=2}^\infty(1-q^{-i})} \Biggl(1+ \sum_{l=2}^\infty
\frac{1}{q^{l(l+1)/2}} \Biggr) \le(1.732) (1.142) \leq2.
\]
The upper bound on the first factor used the second assertion of Lemma
\ref{bound}. Indeed,
\[
\prod_{i \geq2} \bigl(1-q^{-i}\bigr) \geq
\frac
{1-1/q-1/q^2+1/q^5+1/q^7-1/q^{12}-1/q^{15}}{1-1/q}.
\]
The upper bound on the second factor used that
\begin{eqnarray*}
1+ \sum_{l=2}^\infty\frac{1}{q^{l(l+1)/2}}
&\le& 1 + 1/2^3 + 1/2^6 + 1/2^{10} + \sum
_{l=5}^\infty\frac{1}{2^{10+l}}
\\
&=& 1 + 1/2^3 + 1/2^6 + 1/2^{10}
+2/2^{15}\leq1.142.
\end{eqnarray*}\upqed
\end{pf}

We now present the proof of Theorem~\ref{symmetricnonzerodiagthm}.

\begin{pf*}{Proof of Theorem~\ref{symmetricnonzerodiagthm}}
For the lower bound, one computes from the formula for
$p_{0,n}$ in (\ref{pknzddef}), in the case $n=2m$ is even, that
\begin{eqnarray*}
p_{0,n} &=& \frac{N(n,n)}{q^{{n+1 \choose2}}} = (1-1/q) \bigl(1-1/q^2
\bigr) \cdots\bigl(1-1/q^n\bigr) \prod_{i=1}^m
\frac{1}{1-q^{-2i}}
\\
&=& (1-1/q) \bigl(1-1/q^3\bigr)\cdots\bigl(1-1/q^{n-1}
\bigr).
\end{eqnarray*}
Thus, the total variation distance between ${\mathcal Q}_{q,n}$ and
${\mathcal Q}_q$ is at least
\begin{eqnarray*}
& & \frac{1}{2} [p_{0,n}-p_0]
\nonumber
\\
&&\qquad \geq \frac{1}{2} \biggl[ \biggl(1-\frac{1}{q}\biggr) \biggl(1-
\frac{1}{q^3}\biggr) \cdots\biggl(1-\frac{1}{q^{n-1}}\biggr)
\\
&&\hspace*{44.5pt}{} - \biggl(1-
\frac{1}{q}\biggr) \biggl(1-\frac{1}{q^3}\biggr) \cdots\biggl(1-
\frac{1}{q^{n+1}}\biggr) \biggr]
\nonumber
\\
&&\qquad  =  \frac{1}{2 q^{n+1}} (1-1/q) \bigl(1-1/q^3\bigr) \cdots
\bigl(1-1/q^{n-1}\bigr)
\nonumber
\\
&&\qquad \geq \frac{1}{2q^{n+1}} \bigl(1-1/q-1/q^{3}\bigr)
\nonumber
\\
&&\qquad  \geq \frac{0.18}{q^{n+1}}.
\nonumber
\end{eqnarray*}
The second inequality used Lemma~\ref{bound}, and the final inequality
that $q \ge2$.

When $n=2m+1$ is odd, we obtain similarly that
\begin{eqnarray*}
& & \frac{1}{2} [p_{0,n}-p_0]
\nonumber
\\
&&\qquad \geq \frac{1}{2} \biggl[ \biggl(1-\frac{1}{q}\biggr) \biggl(1-
\frac{1}{q^3}\biggr) \cdots\biggl(1-\frac{1}{q^n}\biggr)
\\
&&\hspace*{44.5pt}{} - \biggl(1-
\frac{1}{q}\biggr) \biggl(1-\frac{1}{q^3}\biggr) \cdots\biggl(1-
\frac
{1}{q^{n+2}}\biggr) \biggr]
\nonumber
\\
&&\qquad  =  \frac{1}{2 q^{n+2}} (1-1/q) \bigl(1-1/q^3\bigr) \cdots
\bigl(1-1/q^n\bigr)
\nonumber
\\
&&\qquad  \geq \frac{1}{2q^{n+2}} \bigl(1-1/q-1/q^{3}\bigr)
\nonumber
\\
&&\qquad  \geq \frac{0.18}{q^{n+2}}.
\nonumber
\end{eqnarray*}

To prove the upper bound, for any $A \subset\mathbb{N}_0$ we have
\begin{eqnarray*}
&& \bigl|P(Q_n \in A)-P(Q \in A)\bigr|
\\
&&\qquad =  \bigl|E\bigl[h_A(Q_n)\bigr]-{\mathcal
Q}_qh_A\bigr|
\\
&&\qquad  =  \bigl|E\bigl[f_A(Q_n+1) - \bigl(q^{Q_n}-1
\bigr)f_A(Q_n)\bigr]\bigr|
\\
&&\qquad  =  \bigl|E\bigl[{\mathbf1}_{n-Q_n}q^{-(n-Q_n)}f_A(Q_n+1)
\bigr]\bigr|
\\
&&\qquad  \le {\mathbf1}_n q^{-n}\bigl|f_A(1)\bigr|P(Q_n=0)
\\
&&\quad\qquad{} + \bigl|E\bigl[{\mathbf1}_{n-Q_n}q^{-(n-Q_n)}f_A(Q_n+1){
\mathbf1}(Q_n \ge1)\bigr]\bigr|
\\
&&\qquad \le {\mathbf1}_n q^{-n}\bigl|f_A(1)\bigr|+ E\bigl[{
\mathbf1}_{n-Q_n} q^{-(n-Q_n)}{\mathbf1}(Q_n \ge1)\bigr]\sup
_{k \ge2}\bigl|f_A(k)\bigr|
\\
&&\qquad \le {\mathbf1}_n q^{-n}\bigl|f_A(1)\bigr|+E\bigl[{
\mathbf1}_{n-Q_n} q^{-(n-Q_n)}\bigr]\sup_{k
\ge2}\bigl|f_A(k)\bigr|
\\
&&\qquad = {\mathbf1}_n q^{-n}\bigl|f_A(1)\bigr|+
q^{-n}\sup_{k \ge2}\bigl|f_A(k)\bigr|
\\
&&\qquad  \le {\mathbf1}_n q^{-n} \biggl( \frac{1}{q} +
\frac{1}{q^3} \biggr)+ q^{-n} \biggl( \frac{2}{q^2} \biggr)
\end{eqnarray*}
and the result easily follows. The last two steps used Lemmas~\ref{mom-symmetricnzd-even} and~\ref{solnzdsymmetric}, respectively.
\end{pf*}

\section{Symmetric matrices over finite fields with zero diagonal}\label{0diag}

This section treats the rank distribution (\ref{Nn,rdef}), (\ref
{pkn-odd5def}) of a random symmetric matrix with zero diagonal over a
finite field $\mathbb{F}_q$, when $q$ is a power of 2. Such matrices
were termed ``symplectic'' in \cite{MS}, which studied their rank
distribution in the context of coding theory. We remark that by \cite
{C2} and elementary manipulations, the quantity $N(n,2h)$ defined in
(\ref{Nn,rdef}) below is also equal to the number of $n \times n$
skew-symmetric matrices of rank $2h$ (where now $q$ is odd), so our
results also apply in that context. We also mention that the two
limiting distributions studied in this section arise in the work of the
number theorist Swinnerton--Dyer on 2-Selmer groups \cite{Sw}. We
consider the cases where $n$ is even and odd separately.

\subsection{Case of $n$ even} 

Throughout this subsection, let $n=2m$, an even, nonnegative integer,
and with $q \ge2$, let ${\mathcal Q}_q$ be the distribution on
$\mathbb{N}_0$ with mass function
%
\begin{equation}
\label{pkdef} p_k = \prod_{i \ge1, \;i\; \mathrm{odd}}
\bigl(1-1/q^i\bigr) \frac{q^{2k}}{\prod_{i=1}^{2k}(q^i-1)}.
\end{equation}
For $n \in\mathbb{N}_0$, let ${\mathcal Q}_{q,n}$ be the distribution
on $U_m=\{0,\ldots,m\}$ with mass function
%
\begin{equation}
\label{Nn,rdef} \qquad p_{k,n}=\frac{N(n,n-2k)}{q^{n \choose2}} \qquad\mbox{where }
N(n,2h) = \prod_{i=1}^h
\frac{q^{2i-2}}{q^{2i}-1} \prod_{i=0}^{2h-1}
\bigl(q^{n-i}-1\bigr).
\end{equation}

\begin{theorem} \label{symmetric-eventhm}
We have that
\[
\frac{0.18}{q^{n+1}} \leq\|{\mathcal Q}_{q,n}-{\mathcal
Q}_q\|_{\mathrm{TV}} \leq\frac{1.5}{q^{n+1}}.
\]
\end{theorem}

We begin the proof of Theorem~\ref{symmetric-eventhm} by developing
characterizations of the two distributions of interest.

%
\begin{lemma} 
If $Q$ has the ${\mathcal Q}_q$ distribution, then
\[
E\bigl[q^2f(Q+1)\bigr]=E\bigl[\bigl(q^{2Q-1}-1\bigr)
\bigl(q^{2Q}-1\bigr)f(Q)\bigr]
\]
for all functions $f$ for which these expectations exist.

If $Q_n$ has the ${\mathcal Q}_{q,n}$ distribution then
%
\begin{eqnarray}\label{steineqQnsymmetric}
&& E\bigl[\bigl(q^2-q^{-2(m-Q_n-1)}
\bigr)f(Q_n+1)\bigr]
\nonumber\\[-8pt]\\[-8pt]
&&\qquad =E\bigl[\bigl(q^{2Q_n-1}-1\bigr)
\bigl(q^{2Q_n}-1\bigr)f(Q_n)\bigr]\nonumber
\end{eqnarray}
for all functions $f$ for which these expectations exist.
\end{lemma}

\begin{pf} By taking ratios in (\ref{pkdef}), we obtain that for $k
\in
\mathbb{N}_0$
\[
\frac{p_{k-1}}{p_k}=\frac{(q^{2k-1}-1)(q^{2k}-1)}{q^2}.
\]
Setting $a(k)=q^2$ and $b(k)=(q^{2k-1}-1)(q^{2k}-1)$, applying Lemma
\ref{abcharlem} yields the first result.

Similarly, the second claim can be shown using Lemma~\ref{abcharlem}
and (\ref{Nn,rdef}) to yield
\[
\frac{p_{k-1,n}}{p_{k,n}}= \frac{N(2m,2(m-k+1))}{N(2m,2(m-k))}= \frac
{(q^{2k-1}-1)(q^{2k}-1)}{q^2-q^{-2(m-k)}},
\]
upon setting $a(k)=q^2-q^{-2(m-k)}$ and $b(k)=(q^{2k-1}-1)(q^{2k}-1)$,
noting that $a(m+1)=0$.
\end{pf}

\begin{lemma} \label{mom}
If $Q_n$ has distribution ${\mathcal Q}_{q,n}$, then
\[
Eq^{2Q_n}=q+1-q^{-n+1}.
\]
\end{lemma}

\begin{pf} For $k$ any integer, letting $f(x)=q^{kx}$ in (\ref
{steineqQnsymmetric}) yields
\[
E\bigl[\bigl(q^2-q^{-2(m-Q_n-1)}\bigr)q^{k(Q_n+1)}\bigr]=E
\bigl[\bigl(q^{2Q_n-1}-1\bigr) \bigl(q^{2Q_n}-1\bigr)q^{kQ_n}
\bigr].
\]
Setting $c_k=Eq^{kQ_n}$, this identity yields
\[
q^{-1}c_{k+4}-\bigl(1+q^{-1}- q^{-2m+2+k}
\bigr)c_{k+2}+\bigl(1-q^{k+2}\bigr)c_k = 0.
\]
Substituting $k=-2$ and using that $c_0=1$ we obtain
\[
q^{-1}c_2-\bigl(1+q^{-1}- q^{-2m}\bigr)
= 0,
\]
so that
\[
c_2=q\bigl(1+q^{-1}- q^{-2m}\bigr)=q+1-q^{-2m+1}.
\]\upqed
\end{pf}

In the remainder of this subsection we consider the Stein equation
(\ref{receq}) for the target distribution ${\mathcal Q}_q$ with
\[
a(k)=q^2 \quad\mbox{and} \quad b(k)=\bigl(q^{2k-1}-1
\bigr) \bigl(q^{2k}-1\bigr)
\]
and for $A \subset\mathbb{N}_0$ we let $f_A$ denote the solution
(\ref{steinsol}) when $h(k)={\mathbf1}(k \in A)$.

\begin{lemma} \label{solbnds-evensymmetric}
The function $f_A$ satisfies
\[
\sup_{A \subset\mathbb{N}_0}\bigl|f_A(1)\bigr| \le\frac{1}{q^3}+
\frac{1}{q^5} \quad\mbox{and} \quad\sup_{A \subset\mathbb{N}_0, k \ge2}\bigl|f_A(k)\bigr|
\le\frac{1.31}{q^7}.
\]
\end{lemma}

\begin{pf} By Lemma~\ref{fA1boundlem}, for all $A \subset\mathbb{N}_0$,
\begin{eqnarray*}
\bigl|f_A(1)\bigr| & \le& \frac{P(Q \ge1)}{q^2}
\\
& = & \frac{1-p_0}{q^2}
\\
& = & \frac{1}{q^2} \biggl( 1 -\prod_{i \ge1, \;i\; \mathrm{odd}}
\biggl(1-\frac{1}{q^i}\biggr) \biggr)
\\
& \le& \frac{1}{q^2} \biggl(\frac{1}{q}+\frac{1}{q^3} \biggr)
\\
& = & \frac{1}{q^3}+\frac{1}{q^5},
\end{eqnarray*}
where the second inequality used Lemma~\ref{bound}.

For $k \ge1$, by (\ref{prodUUcbound}) and (\ref{pkdef}),
\begin{eqnarray*}
\bigl|f_A(k+1)\bigr| &\le& \frac{P(Q \in U_k^c)}{q^2 p_k}
\\
&=& \frac{\prod_{i=1}^{2k}(q^i-1)}{q^{2k+2}} \sum_{l=k+1}^\infty
\frac{q^{2l}}{\prod_{i=1}^{2l}(q^i-1)}
\\
&=& \frac{1}{q^2} \sum_{l=k+1}^\infty
\frac{q^{2(l-k)}}{\prod_{i=2k+1}^{2l}(q^i-1)}
\\
&=& \frac{1}{q^2} \sum_{l=k+1}^\infty
\frac
{q^{l-k}}{q^{2(l^2-k^2)}\prod_{i=2k+1}^{2l}(1-q^{-i})}
\\
&\le& \frac{1}{q^2\prod_{i=2k+1}^\infty(1-q^{-i})} \sum_{l=k+1}^\infty
\frac{q^{l-k}}{q^{2(l^2-k^2)}}
\\
&=& \frac{1}{q^2 \prod_{i=2k+1}^\infty(1-q^{-i})} \Biggl( \frac
{1}{q^{4k+1}} + \sum
_{l=2}^\infty\frac{1}{q^{2l^2+4lk-l}} \Biggr)
\\
&\le& \frac{1}{q^2 \prod_{i=2k+1}^\infty(1-q^{-i})} \Biggl( \frac
{1}{q^{4k+1}} + \frac{1}{q^{8k}}\sum
_{l=2}^\infty\frac
{1}{q^{2l^2-l}} \Biggr).
\end{eqnarray*}

Hence, for all $k \ge1$ we obtain
\[
\bigl|f_A(k+1)\bigr| \le\frac{1}{q^7\prod_{i=3}^\infty(1-q^{-i})} \Biggl(1+ \frac
{1}{q^3}\sum
_{l=2}^\infty\frac{1}{q^{2l^2-l}} \Biggr)
\]
and the proof is now completed using the fact that for all $q \ge2$
\[
\frac{1}{\prod_{i=3}^\infty(1-q^{-i})} \Biggl(1+ \frac{1}{q^3}\sum
_{l=2}^\infty\frac{1}{q^{2l^2-l}} \Biggr) \le(1.29854)
(1.002) \leq1.31.
\]
The upper bound on the first factor used part 2 of Lemma~\ref{bound}.
The upper bound on the second factor used that
\begin{eqnarray*}
1+ \frac{1}{q^3}\sum_{l=2}^\infty
\frac{1}{q^{2l^2-l}} &\le& 1+\frac
{1}{2^3} \Biggl( \frac{1}{2^6} + \sum
_{l=3}^\infty\frac
{1}{2^{12+l}} \Biggr)
\\
&=&  1+\frac{1}{8} \biggl( \frac{1}{2^6} + \frac
{2}{2^{15}} \biggr)
\leq1.002.
\end{eqnarray*}\upqed
\end{pf}

We now present the proof of Theorem~\ref{symmetric-eventhm}.

\begin{pf*}{Proof of Theorem~\ref{symmetric-eventhm}}
From the formula for $p_{0,n}$, one has that
\[
p_{0,n} = \frac{N(n,n)}{q^{{n \choose2}}} = (1-1/q) \bigl(1-1/q^3
\bigr) \cdots\bigl(1-1/q^{n-1}\bigr).
\]
The argument in the proof of Theorem~\ref{symmetricnonzerodiagthm}
now shows that total variation distance between ${\mathcal Q}_{q,n}$
and ${\mathcal Q}_q$ is at least $0.18/q^{n+1}$.

For the upper bound, arguing as in the proof of Theorem~\ref{symmetricnonzerodiagthm} we obtain
\begin{eqnarray*}
&& \bigl|P(Q_n \in A)-P(Q \in A)\bigr|
\\
&&\qquad  =  \bigl|E\bigl[h_A(Q_n)\bigr]-{\mathcal
Q}_qh_A\bigr|
\\
&&\qquad  =  \bigl|E\bigl[q^2f_A(Q_n+1) -
\bigl(q^{2Q_n-1}-1\bigr) \bigl(q^{2Q_n}-1\bigr) f_A(Q_n)
\bigr]\bigr|
\\
&&\qquad  =  \bigl|E\bigl[q^{-2(m-Q_n-1)}f_A(Q_n+1)\bigr]\bigr|
\\
&&\qquad \le \bigl|q^{-2(m-1)}f_A(1)\bigr|P(Q_n=0)
\\
&&\qquad\quad{} +\bigl|E\bigl[q^{-2(m-Q_n-1)}f_A(Q_n+1){
\mathbf1}(Q_n \ge1)\bigr]\bigr|
\\
&&\qquad \le q^{-2(m-1)}\bigl|f_A(1)\bigr|+E\bigl[q^{-2(m-Q_n-1)}{
\mathbf1}(Q_n \ge1)\bigr]\sup_{k
\ge2}\bigl|f_A(k)\bigr|
\\
&&\qquad \le q^{-2(m-1)}\bigl|f_A(1)\bigr|+E\bigl[q^{-2(m-Q_n-1)}\bigr]\sup
_{k \ge2}\bigl|f_A(k)\bigr|
\\
&&\qquad = q^{-2(m-1)}\bigl|f_A(1)\bigr|+ q^{-2(m-1)}
\bigl(q+1-q^{-2m+1}\bigr)\sup_{k \ge
2}\bigl|f_A(k)\bigr|
\\
&&\qquad \le q^{-2(m-1)}\bigl|f_A(1)\bigr|+ q^{-2(m-1)}(q+1)\sup
_{k \ge2}\bigl|f_A(k)\bigr|
\\
&&\qquad \le q^{-n+2} \biggl(\frac{1}{q^3}+\frac{1}{q^5} \biggr)+
1.31\bigl(q^{-n+3}+q^{-n+2}\bigr)\frac{1}{q^7}
\\
&&\qquad = q^{-(n+1)}+q^{-(n+3)}+1.31q^{-(n+4)}+
1.31q^{-(n+5)}
\\
&&\qquad \le 1.5 q^{-(n+1)}
\end{eqnarray*}
as claimed. Note that Lemma~\ref{mom} was used in the fourth equality,
and Lemma~\ref{solbnds-evensymmetric} in the second to last inequality.
\end{pf*}

\subsection{Case of $n$ odd} 

Throughout this subsection, let $n=2m+1$, a positive, odd integer, and
with $q \ge2$, let ${\mathcal Q}_q$ be the distribution on $\mathbb
{N}_0$ with mass function
%
\begin{equation}
\label{pk-odddef} p_k = \prod_{i \ge1, \;i\; \mathrm{odd}}
\bigl(1-1/q^i\bigr) \frac{q^{2k+1}}{\prod_{i=1}^{2k+1}(q^i-1)}.
\end{equation}
For $n \in\mathbb{N}_0$ let ${\mathcal Q}_{q,n}$ be the distribution
on $\{0,\ldots,m\}$ with mass function
%
\begin{equation}
\label{pkn-odd5def} p_{k,n}=\frac{N(n,n-1-2k)}{q^{n \choose2}},
\end{equation}
where $N(n,2h)$ is given in (\ref{Nn,rdef}).

Our main result is the following theorem.

\begin{theorem} \label{symmetric-oddthm}
We have that
\[
\frac{0.37}{q^{n+2}} \leq\|{\mathcal Q}_{q,n}-{\mathcal
Q}_q\|_{\mathrm{TV}} \leq\frac{2.2}{q^{n+2}}.
\]
\end{theorem}

We again begin by developing characterizing equations for the
distributions under study.

%
\begin{lemma} 
If $Q$ has the ${\mathcal Q}_q$ distribution, then
\[
E\bigl[q^2f(Q+1)\bigr]=E\bigl[\bigl(q^{2Q+1}-1\bigr)
\bigl(q^{2Q}-1\bigr)f(Q)\bigr]
\]
for all functions $f$ for which these expectations exist.

If $Q_n$ has the ${\mathcal Q}_{q,n}$ distribution then
%
\begin{eqnarray}\label{steineqQn-oddsymmetric}
&& E\bigl[\bigl(q^2-q^{-2(m-Q_n-1)}
\bigr)f(Q_n+1)\bigr]
\nonumber\\[-8pt]\\[-8pt]
&&\qquad =E\bigl[\bigl(q^{2Q_n+1}-1\bigr)
\bigl(q^{2Q_n}-1\bigr)f(Q_n)\bigr]\nonumber
\end{eqnarray}
for all functions $f$ for which these expectations exist.
\end{lemma}

\begin{pf} By taking ratios in (\ref{pk-odddef}), we obtain
\[
\frac{p_{k-1}}{p_k}=\frac{(q^{2k+1}-1)(q^{2k}-1)}{q^2}.
\]
Setting $a(k)=q^2$ and $b(k)=(q^{2k+1}-1)(q^{2k}-1)$, Lemma~\ref{abcharlem} yields the first claim. Similarly, the second can be
shown by applying (\ref{pkn-odd5def}) to yield
\[
\frac{p_{k-1,n}}{p_{k,n}}= \frac{N(n,2(m-k+1))}{N(n,2(m-k))}= \frac
{(q^{2k+1}-1)(q^{2k}-1)}{q^2-q^{-2(m-k)}}
\]
and then invoking Lemma~\ref{abcharlem} with $a(k)=q^2-q^{-2(m-k)}$
and $b(k)=(q^{2k+1}-1)(q^{2k}-1)$, noting $a(m+1)=0$.
\end{pf}

\begin{lemma} \label{mom-symmetric-even}
If $Q_n$ has distribution ${\mathcal Q}_{q,n}$ then
\[
Eq^{2Q_n}=1+q^{-1}-q^{-n}.
\]
\end{lemma}

\begin{pf} For $k$ any integer, letting $f(x)=q^{kx}$ in (\ref
{steineqQn-oddsymmetric}) yields
\[
E\bigl[\bigl(q^2-q^{-2(m-Q_n-1)}\bigr)q^{k(Q_n+1)}\bigr]=E
\bigl[\bigl(q^{2Q_n+1}-1\bigr) \bigl(q^{2Q_n}-1
\bigr)q^{kQ_n}\bigr].
\]
Setting $c_k=E[q^{kQ_n}]$, this identity yields
\[
q c_{k+4}-\bigl(1+q- q^{-2m+2+k}\bigr)c_{k+2}+
\bigl(1-q^{k+2}\bigr)c_k = 0.
\]
Substituting $k=-2$ and using that $c_0=1$ we obtain
\[
q c_2-\bigl(1+q- q^{-2m}\bigr) = 0,
\]
so that
\[
c_2=q^{-1}\bigl(1+q- q^{-2m}\bigr)= 1+
q^{-1}-q^{-2m-1}.
\]\upqed
\end{pf}

In the remainder of this subsection, we consider the Stein equation
(\ref{receq}) for the target distribution ${\mathcal Q}_q$ with
\[
a(k)=q^2 \quad\mbox{and} \quad b(k)=\bigl(q^{2k+1}-1
\bigr) \bigl(q^{2k}-1\bigr)
\]
and for $A \subset\mathbb{N}_0$ we let $f_A$ denote the solution
(\ref{steinsol}) when $h(k)={\mathbf1}(k \in A)$.

\begin{lemma} \label{solbnds-oddsymmetric}
The function $f_A$ satisfies
\[
\sup_{A \subset\mathbb{N}_0}\bigl|f_A(1)\bigr| \le\frac{2}{q^5} \quad
\mbox{and} \quad\sup_{A \subset\mathbb{N}_0, k \ge2}\bigl|f_A(k)\bigr| \le
\frac{1.14}{q^9}.
\]
\end{lemma}

\begin{pf} By Lemma~\ref{fA1boundlem}, for all $A \subset\mathbb{N}_0$,
\begin{eqnarray*}
\bigl|f_A(1)\bigr| & \le& \frac{P(Q \ge1)}{q^2}
\\
& = & \frac{1-p_0}{q^2}
\\
& = & \frac{1}{q^2} \biggl( 1 - \frac{q}{q-1}\prod
_{i \ge1, \;i\; \mathrm{odd}}\biggl(1-\frac{1}{q^i}\biggr) \biggr)
\\
& = & \frac{1}{q^2} \biggl( 1 - \prod_{i \ge3, \;i\; \mathrm{odd}}
\biggl(1-\frac{1}{q^i}\biggr) \biggr)
\\
& \le& \frac{1}{q^2} \biggl(\frac{2}{q^3} \biggr) =
\frac{2}{q^5},
\end{eqnarray*}
where the second inequality used Lemma~\ref{bound}.

For $k \ge1$, by (\ref{prodUUcbound}) and (\ref{pk-odddef}),
\begin{eqnarray*}
\bigl|f_A(k+1)\bigr| &\le& \frac{P(Q \in U_k^c)}{q^2 p_k}
\\
&=& \frac{\prod_{i=1}^{2k+1}(q^i-1)}{q^{2k+3}} \sum_{l=k+1}^\infty
\frac{q^{2l+1}}{\prod_{i=1}^{2l+1}(q^i-1)}
\\
&=& \frac{1}{q^2} \sum_{l=k+1}^\infty
\frac{q^{2(l-k)}}{\prod_{i=2k+2}^{2l+1}(q^i-1)}
\\
&=& \frac{1}{q^2} \sum_{l=k+1}^\infty
\frac
{1}{q^{2(l^2-k^2)+(l-k)}\prod_{i=2k+2}^{2l+1}(1-q^{-i})}
\\
&\le& \frac{1}{q^2\prod_{i=2k+2}^\infty(1-q^{-i})} \sum_{l=k+1}^\infty
\frac{1}{q^{2(l^2-k^2)+(l-k)}}
\\
&=& \frac{1}{q^2 \prod_{i=2k+2}^\infty(1-q^{-i})} \Biggl( \frac
{1}{q^{4k+3}} + \sum
_{l=2}^\infty\frac{1}{q^{2l^2+4lk+l}} \Biggr)
\\
&\le& \frac{1}{q^2 \prod_{i=2k+2}^\infty(1-q^{-i})} \Biggl( \frac
{1}{q^{4k+3}} + \frac{1}{q^{8k}}\sum
_{l=2}^\infty\frac
{1}{q^{2l^2+l}} \Biggr).
\end{eqnarray*}

Hence, for all $k \ge1$, we obtain
\[
\bigl|f_A(k+1)\bigr| \le\frac{1}{q^9\prod_{i=4}^\infty(1-q^{-i})} \Biggl(1+ \frac
{1}{q}\sum
_{l=2}^\infty\frac{1}{q^{2l^2+l}} \Biggr)
\]
and the proof is now completed by using the fact that for all $q \ge2$,
\[
\frac{1}{\prod_{i=4}^\infty(1-q^{-i})} \Biggl(1+ \frac{1}{q}\sum
_{l=2}^\infty\frac{1}{q^{2l^2+l}} \Biggr) \le(1.137)
(1.0005) \leq1.14.
\]
The inequality $\prod_{i=4}^{\infty} (1-q^{-i})^{-1} \le1.137$ is
obtained by applying part 2 of Lemma~\ref{bound}. We also used that
\begin{eqnarray*}
1+ \frac{1}{q}\sum_{l=2}^\infty
\frac{1}{q^{2l^2+l}} &\leq& 1 + \frac
{1}{2} \Biggl( \frac{1}{2^{10}} + \sum
_{l=3}^\infty\frac
{1}{2^{18+l}} \Biggr)
\\
&=&  1 + \frac{1}{2} \biggl( \frac{1}{2^{10}} + \frac{2}{2^{21}} \biggr)
\leq 1.0005.
\end{eqnarray*}\upqed
\end{pf}

We now present the proof of Theorem~\ref{symmetric-oddthm}.

\begin{pf*}{Proof of Theorem~\ref{symmetric-oddthm}}
From the formula (\ref{pkn-odd5def}) for $p_{0,n}$, we obtain
\begin{eqnarray*}
p_{0,n} &=& \frac{N(n,n-1)}{q^{{n \choose2}}}
\\
& =& (1-1/q) \bigl(1-1/q^3
\bigr) \cdots\bigl(1-1/q^{n}\bigr) \frac{q}{q-1}.
\end{eqnarray*}
Thus, now applying (\ref{pk-odddef}), the total variation distance
between ${\mathcal Q}_{q,n}$ and ${\mathcal Q}_q$ is at least
\begin{eqnarray*}
& & \frac{1}{2} [p_{0,n}-p_0]
\\
&&\qquad \geq \frac{q}{2(q-1)}
\biggl[ \biggl(1-\frac{1}{q}\biggr) \biggl(1-\frac{1}{q^3}
\biggr) \cdots\biggl(1-\frac
{1}{q^{n}}\biggr)
\\
&&\hspace*{77pt}{} - \biggl(1-\frac{1}{q}
\biggr) \biggl(1-\frac{1}{q^3}\biggr) \cdots\biggl(1-\frac
{1}{q^{n+2}}
\biggr) \biggr]
\\
&&\qquad =  \frac{q}{2(q-1) q^{n+2}} (1-1/q) \bigl(1-1/q^3\bigr) \cdots
\bigl(1-1/q^{n}\bigr)
\\
&&\qquad  = \frac{1}{2 q^{n+2}} \bigl(1-1/q^3\bigr) \cdots
\bigl(1-1/q^n\bigr)
\\
&&\qquad  \geq \frac{1}{2q^{n+2}} \bigl(1-2/q^{3}\bigr)
\\
&&\qquad  \geq \frac{0.37}{q^{n+2}}.
\end{eqnarray*}
The second inequality used the fourth claim of Lemma~\ref{bound}.

Arguing as for the proof of Theorem~\ref{symmetricnonzerodiagthm},
for any $A \subset\mathbb{N}_0$ we have
\begin{eqnarray*}
&&\bigl|P(Q_n \in A)-P(Q \in A)\bigr|
\\
&&\qquad  =  \bigl|E\bigl[h_A(Q_n)\bigr]-{\mathcal
Q}_qh_A\bigr|
\\
&&\qquad =  \bigl|E\bigl[q^2f_A(Q_n+1) -
\bigl(q^{2Q_n+1}-1\bigr) \bigl(q^{2Q_n}-1\bigr)
f_A(Q_n)\bigr]\bigr|
\\
&&\qquad =  \bigl|E\bigl[q^{-2(m-Q_n-1)}f_A(Q_n+1)\bigr]\bigr|
\\
&&\qquad \le q^{-2(m-1)}\bigl|f_A(1)\bigr|P(Q_n=0)
\\
&&\quad\qquad{} +\bigl|E\bigl[q^{-2(m-Q_n-1)}f_A(Q_n+1){
\mathbf1}(Q_n \ge1)\bigr]\bigr|
\\
&&\qquad \le q^{-2(m-1)}\bigl|f_A(1)\bigr|+E\bigl[q^{-2(m-Q_n-1)}{
\mathbf1}(Q_n \ge1)\bigr]\sup_{k
\ge2}\bigl|f_A(k)\bigr|
\\
&&\qquad \le q^{-2(m-1)}\bigl|f_A(1)\bigr|+E\bigl[q^{-2(m-Q_n-1)}\bigr]\sup
_{k \ge2}\bigl|f_A(k)\bigr|
\\
&&\qquad = q^{-2(m-1)}\bigl|f_A(1)\bigr|+ q^{-2(m-1)}
\bigl(1+q^{-1}-q^{-n}\bigr)\sup_{k \ge
2}\bigl|f_A(k)\bigr|
\\
&&\qquad \le q^{-2(m-1)}\bigl|f_A(1)\bigr|+ q^{-2(m-1)}
\bigl(1+q^{-1}\bigr)\sup_{k \ge
2}\bigl|f_A(k)\bigr|
\\
&&\qquad \le q^{-n+3}\frac{2}{q^5} + 1.14\bigl(q^{-n+3}+q^{-n+2}
\bigr)\frac{1}{q^9}
\\
&&\qquad = 2 q^{-(n+2)}+1.14q^{-(n+6)}+ 1.14q^{-(n+7)}
\\
&&\qquad \le 2.2 q^{-(n+2)}
\end{eqnarray*}
as claimed, where we have applied Lemmas~\ref{mom-symmetric-even} and
\ref{solbnds-oddsymmetric} in the second to last equality, and
inequality, respectively.
\end{pf*}

\section{Skew centrosymmetric matrices over finite fields} \label{skew}
An $n \times n$ matrix $A$ is called skew centrosymmetric if $A_{ij}=-A_{ji}$
and $A_{ij}=A_{n+1-j,n+1-i}$. This section studies the rank
distributions (\ref{pknskeweven}) and (\ref{pknskewodd}) of a
randomly chosen
skew centrosymmetric matrix with entries in $\mathbb{F}_q$ for $q$ odd.

Suppose that $n$ is even. Waterhouse \cite{W} shows that the total
number of skew centrosymmetric
matrices is $q^{(n/2)^2}$, that all such matrices have even rank, and
that the proportion of $n \times n$
skew centrosymmetric matrices of rank $n-2k$ is equal to
%
\begin{eqnarray}\label{pknskeweven}
p_{k,n} = \frac{N(n,n-2k)}{q^{(n/2)^2}},
\nonumber\\[-8pt]\\[-8pt]
\eqntext{\displaystyle\mbox{where } N(n,2h) = \prod_{j=0}^{n/2-h-1}
\frac{q^{n/2}-q^j}{q^{n/2-h}-q^j} \prod_{i=0}^{h-1}
\bigl(q^{n/2}-q^i\bigr).}
\end{eqnarray}
We claim that $p_{k,n}$ in (\ref{pknskeweven}) is exactly equal to
the probability that a uniformly chosen
$n/2 \times n/2$ random matrix with entries from $\mathbb{F}_q$ has
rank $n/2-k$. Indeed, pulling out factors of $q$, one can write (\ref
{pknskeweven}) as
\[
\frac{1}{q^{k^2}} \prod_{j=0}^{k-1}
\biggl( \frac
{1-q^{j-n/2}}{1-q^{j-k}} \biggr) \prod_{j=k+1}^{n/2}
\bigl(1-q^{-j}\bigr).
\]
Comparing this expression with (\ref{defQqnrect}) for the case $m=0$
with $n$ replaced by $n/2$ shows that it is sufficient to prove that
\[
\prod_{j=0}^{k-1} \biggl(
\frac{1-q^{j-n/2}}{1-q^{j-k}} \biggr) = \frac{\prod_{j=k+1}^{n/2}
(1-q^{-j})}{\prod_{j=1}^{n/2-k} (1-q^{-j})}.
\]
This identity holds since both
\[
\prod_{j=0}^{k-1} \bigl(1-q^{j-n/2}
\bigr) \prod_{j=1}^{n/2-k}
\bigl(1-q^{-j}\bigr)
\]
and
\[
\prod
_{j=0}^{k-1} \bigl(1-q^{j-k}\bigr) \prod
_{j=k+1}^{n/2} \bigl(1-q^{-j}\bigr)
\]
are equal to $\prod_{i=1}^{n/2} (1-1/q^i)$. Hence, the following
corollary is immediate from Theorem~\ref{main}.

\begin{cor}
For $q \ge2$, let ${\mathcal Q}_q$ be the distribution (\ref
{defQqrect}) on $\mathbb{N}_0$, specialized to $m=0$. For $n$ even
in $\mathbb{N}_0$, let ${\mathcal Q}_{q,n}$ be the distribution on $\{
0,\ldots,n/2 \}$ with mass function (\ref{pknskeweven}).
Then
\[
\frac{1}{8 q^{n/2+1}} \leq\|{\mathcal Q}_{q,n}-{\mathcal
Q}_q\|_{\mathrm{TV}} \leq\frac{3}{q^{n/2+1}}.
\]
\end{cor}

Now suppose that $n$ is odd. Waterhouse \cite{W} shows that the total
number of skew centrosymmetric
matrices is $q^{(n-1)^2/4+(n-1)/2}$, that all such matrices have even
rank and that the number of $n \times n$
skew centrosymmetric matrices of rank $2h$ is equal to
\[
N(n,2h) = \prod_{j=0}^{(n-1)/2-h}
\frac
{q^{(n-1)/2+1}-q^j}{q^{(n-1)/2+1-h}-q^j} \prod_{i=0}^{h-1}
\bigl(q^{(n-1)/2}-q^i\bigr).
\]
Hence,
%
\begin{equation}
\label{pknskewodd} p_{k,n} = \frac{N(n,n-2k-1)}{q^{(n-1)^2/4 +
(n-1)/2}}, \qquad k \in U_{(n-1)/2} = \bigl\{0,1,\ldots,(n-1)/2\bigr\}
\end{equation}
is the proportion of skew centrosymmetric matrices of
rank $n-2k-1$. The main result in this section is Theorem~\ref{main2},
which provides bounds on the total variation distance between
${\mathcal Q}_{q,n}$, the distribution given in (\ref{pknskewodd}),
and ${\mathcal Q}_q$, given by
%
\begin{equation}
\label{1pkdef} p_k = \frac{\prod_{i \geq1}
(1-1/q^i)}{q^{k^2+k}(1-1/q^{k+1})\prod_{i=1}^k(1-1/q^i)^2}, \qquad k \in \mathbb{N}_0.
\end{equation}

\begin{theorem} \label{main2}
For $n \geq1$ odd, and $q \geq2$, we have that
\[
\frac{1}{4 q^{(n+3)/2}} \leq\|{\mathcal Q}_{q,n}-{\mathcal
Q}_q\|_{\mathrm{TV}} \leq\frac{3}{q^{(n+3)/2}}.
\]
\end{theorem}

We begin with the following characterization lemma.

\begin{lemma} 
If $Q$ has the ${\mathcal Q}_q$ distribution, then
\[
E \bigl[qf(Q+1) \bigr]=E \bigl[\bigl(q^{Q}-1\bigr)
\bigl(q^{Q+1}-1\bigr) f(Q) \bigr]
\]
for all functions $f$ for which these expectations exist.

If $Q_n$ has the ${\mathcal Q}_{q,n}$ distribution, then
%
\begin{equation}\label{Qnq}
E \bigl[ \bigl(q - q^{Q_n+1-(n-1)/2} \bigr) f(Q_n+1)
\bigr]
 =E \bigl[\bigl(q^{Q_n}-1\bigr) \bigl(q^{Q_n+1}-1\bigr)
f(Q_n) \bigr]\hspace*{-30pt}
\end{equation}
for all functions $f$ for which these expectations exist.
\end{lemma}

\begin{pf} For the first assertion, one calculates that
\[
\frac{p_{k-1}}{p_k} = \frac{(q^k-1)(q^{k+1}-1)}{q}, \qquad k \in\mathbb{N}_0.
\]
Taking $a(k)=q$ and $b(k)=(q^k-1)(q^{k+1}-1)$ in Lemma~\ref{abcharlem}, the first assertion follows.

For the second assertion, one calculates that
\[
\frac{p_{k-1,n}}{p_{k,n}} = \frac{N(n,n-2k+1)}{N(n,n-2k-1)} = \frac
{(q^k-1)(q^{k+1}-1)}{q-q^{k-(n-1)/2}}, \qquad k \in
U_{(n-1)/2}.
\]
Taking $a(k)=q-q^{k-(n-1)/2}$ and $b(k)=(q^k-1)(q^{k+1}-1)$, noting
that $a((n-1)/2+1)=0$, the second assertion follows by Lemma~\ref{abcharlem}.
\end{pf}

Lemma~\ref{mom3} calculates the expected value of $q^{Q_n}$.

\begin{lemma} \label{mom3} If $Q_n$ has distribution ${\mathcal
Q}_{q,n}$, then
\[
E\bigl[q^{Q_n}\bigr] = 1 + \frac{1}{q} - \frac{1}{q^{(n+1)/2}}.
\]
\end{lemma}

\begin{pf} Let $c_k = E[q^{k Q_n}]$, and set $f(x)=q^{kx}$ in (\ref
{Qnq}). Elementary manipulations yield the
recurrence
\[
q c_{k+2} = \bigl(q+1-q^{k+1 - (n-1)/2}\bigr) c_{k+1} +
\bigl(q^{k+1}-1\bigr) c_k.
\]
The result now follows by
setting $k=-1$ and using that $c_0=1$.
\end{pf}

In the remainder of this section, we consider the Stein equation (\ref
{receq}) for the target distribution ${\mathcal Q}_q$ with
\[
a(k)=q \quad\mbox{and} \quad b(k)=\bigl(q^k-1\bigr)
\bigl(q^{k+1}-1\bigr)
\]
and for $A \subset\mathbb{N}_0$ we let $f_A$ denote the solution
(\ref{steinsol}) when $h(k)={\mathbf1}(k \in A)$.

\begin{lemma} \label{bound3} The function $f_A$ satisfies
\[
\sup_{A \subset\mathbb{N}_0} \bigl|f_A(k)\bigr| \leq\frac{2}{q^3}.
\]
\end{lemma}

\begin{pf} By Lemma~\ref{fA1boundlem} and (\ref{1pkdef}),
\begin{eqnarray*}
\bigl|f_A(1)\bigr| & \leq& \frac{P(Q \geq1)}{q}
\\
& = & \frac{1-p_0}{q}
\\
& = & \frac{1 - \prod_{i \geq2} (1-1/q^i)}{q}
\\
& \leq& \frac{1- (1-\sum_{i \ge2} 1/q^i)}{q}
\\
& = & \frac{1}{q^3(1-1/q)}
\\
& \leq& 2/q^3,
\end{eqnarray*}
where we have applied (\ref{prodagesuma}) in the second inequality,
and used
that $q \geq2$.

For $k \geq1$, by (\ref{prodUUcbound}),
\begin{eqnarray*}
\bigl|f_A(k+1)\bigr|
&\leq& \frac{P(Q \in U_k^c)}{q p_k}
\\
&=&  q^{k^2+k-1}\bigl(1-1/q^{k+1}\bigr)
\\
&&{}\times  \prod_{i=1}^k\bigl(1-1/q^i
\bigr)^2
\sum_{l=k+1}^{\infty}
\frac{1}{q^{l^2+l}(1-1/q^{l+1})\prod_{i=1}^l(1-1/q^i)^2}
\\
&=&  q^{k^2+k-1} \sum_{l=k+1}^{\infty}
\frac{1}{q^{l^2+l}
(1-1/q^{k+1})(1-1/q^{l+1}) \prod_{j=k+2}^l (1-1/q^j)^2}
\\
&\leq& q^{k^2+k-1} \sum_{l=k+1}^{\infty}
\frac{1}{q^{l^2+l} \prod_{j=k+1}^l (1-1/q^j)^2}
\\
&\leq& \frac{q^{k^2+k-1}}{\prod_{j=k+1}^{\infty} (1-1/q^j)^2} \sum
_{l=k+1}^{\infty}
\frac{1}{q^{l^2+l}}
\\
& \leq& \frac{q^{k^2+k-1}}{(1 - \sum_{j=k+1}^{\infty} 1/q^j)^2} \sum
_{l=k+1}^{\infty}
\frac{1}{q^{l^2+l}}
\\
&=&  \frac{q^{k^2+k-1}}{(1 - (1/({q^k(q-1)})))^2} \sum_{l=1}^{\infty}
\frac{1}{q^{(k+l)^2+k+l}}
\\
&\leq& \frac{4}{q}\sum_{l=1}^{\infty}
\frac{1}{q^{l^2+l+2kl}}
\\
&\le& \frac{4}{q^3}\sum_{l=1}^{\infty}
\frac{1}{q^{l^2+l}}
\\
&\le& \frac{4}{q^3}\sum_{l=1}^\infty
\frac{1}{2^{l+1}}
\\
&=&  \frac{2}{q^3},
\end{eqnarray*}
where (\ref{prodagesuma}) was applied in the fourth inequality.
\end{pf}

We now present the proof of Theorem~\ref{main2}.

\begin{pf*}{Proof of Theorem~\ref{main2}}
From the formula (\ref{pknskewodd}) for $p_{0,n}$, one
computes that
\[
p_{0,n} = \frac{N(n,n-1)}{q^{(n-1)^2/4 + (n-1)/2}} = \bigl(1-1/q^2\bigr)
\bigl(1-1/q^3\bigr) \cdots\bigl(1-1/q^{(n+1)/2}\bigr).
\]
Thus, using (\ref{1pkdef}), the total variation distance between
${\mathcal Q}_{q,n}$ and ${\mathcal Q}_q$ is at least
\begin{eqnarray*}
\frac{1}{2} [p_{0,n} - p_0]
&\geq& \frac{1}{2} \bigl[\bigl(1-1/q^2\bigr)
\bigl(1-1/q^3\bigr) \cdots\bigl(1-1/q^{(n+1)/2}\bigr) \bigr]
\\
&&{} - \frac{1}{2} \bigl[\bigl(1-1/q^2\bigr)
\bigl(1-1/q^3\bigr) \cdots\bigl(1-1/q^{(n+3)/2}\bigr)\bigr]
\\
&=&  \frac{1}{2 q^{(n+3)/2}} \bigl(1-1/q^2\bigr) \bigl(1-1/q^3
\bigr) \cdots\bigl(1-1/q^{(n+1)/2}\bigr).
\end{eqnarray*}
By part 1 of Lemma~\ref{bound},
\begin{eqnarray*}
&& \bigl(1-1/q^2\bigr) \bigl(1-1/q^3\bigr) \cdots
\bigl(1-1/q^{(n+1)/2}\bigr)
\\
&&\qquad =  \frac{(1-1/q)(1-1/q^2) \cdots(1-1/q^{(n+1)/2})}{(1-1/q)}
\\
&&\qquad \geq \frac{1-1/q-1/q^2}{1-1/q}
\\
&&\qquad  \geq 1/2.
\end{eqnarray*}
It follows that the total variation distance between ${\mathcal
Q}_{q,n}$ and ${\mathcal Q}_q$ is at least
$1/(4 q^{(n+3)/2})$.

For the upper bound, arguing as in Theorem~\ref{main},
\begin{eqnarray*}
&& \bigl|P(Q_n \in A) - P(Q \in A)\bigr|
\\
&&\qquad  =  \bigl|E\bigl[h_A(Q_n)\bigr] - Q_q
h_A\bigr|
\\
&&\qquad =  \bigl|E\bigl[q f_A(Q_n+1) - \bigl(q^{Q_n}-1
\bigr) \bigl(q^{Q_n+1}-1\bigr) f_A(Q_n)\bigr]\bigr|
\\
&&\qquad =  \bigl|E\bigl[q^{Q_n+1-(n-1)/2} f_A(Q_n+1)\bigr]\bigr|
\\
&&\qquad \leq \|f_A\| E\bigl[q^{Q_n+1-(n-1)/2}\bigr].
\end{eqnarray*}
By Lemmas~\ref{mom3} and~\ref{bound3}, this quantity is at most
\[
\frac{2}{q^3} q^{1-(n-1)/2} (1+1/q) \leq\frac{3}{q^{(n+3)/2}}.
\]\upqed
\end{pf*}

\section{Hermitian matrices over finite fields} \label{Hermitian}
Let $q$ be odd. Suppose that $\theta\in\mathbb{F}_{q^2}, \theta^2
\in\mathbb{F}_q$, but $\theta\notin\mathbb{F}_q$. Then any $\alpha
\in\mathbb{F}_{q^2}$ can be written $\alpha= a + b \theta$ with
\mbox{$a,b \in\mathbb{F}_q$}. By the conjugate of $\alpha$, we mean
$\bar{\alpha}=a - b \theta$. If $A = (\alpha_{ij})$ is a
square~matrix, $\alpha_{ij} \in\mathbb{F}_{q^2}$, let $A^*=
\overline{A}' = (\bar{\alpha}_{ij})'$, where the prime denotes
transpose. Then $A$ is said to be Hermitian if and only if $A^*=A$.

By \cite{CH}, for $q$ odd the total number of $n \times n$ Hermitian matrices
over $\mathbb{F}_q$ is $q^{n^2}$, and the total number of such
matrices with rank $r$ is
\[
N(n,r) = q^{{r \choose2}} \prod_{i=1}^r
\frac{q^{2n-2(r-i)}-1}{q^i -
(-1)^i}.
\]
Hence, the proportion of such matrices with rank $n-k$ is given by
%
\begin{equation}
\label{pknhermitian} p_{k,n} = \frac{N(n,n-k)}{q^{n^2}}, \qquad k \in
U_n = \{0,\ldots,n\}.
\end{equation}
In this section we compute total variation bounds between the
distribution (\ref{pknhermitian}), denoted ${\mathcal Q}_{q,n}$, and
the distribution
%
%
\begin{equation}
\label{pkhermitian} p_k = \prod_{i\; \mathrm{odd}}
\frac{1}{1+1/q^i} \cdot\frac{1}{q^{k^2}\prod_{i=1}^k (1-1/q^{2i})},
\end{equation}
which we denote here by ${\mathcal Q}_q$.

\begin{remark*} The distribution (\ref{pkhermitian}) also arises as a
limiting law
in the study of the dimension of the fixed space of a random
element of the finite unitary group $U(n,q)$. More precisely, the paper
\cite{RS} proves that for $k$ fixed, the chance that
a uniformly chosen random element of $U(n,q)$ has a $k$ dimensional
fixed space tends to $p_k$ as $n \rightarrow\infty$.
See \cite{F3} for another proof.

The main theorem of this section is the following result.
\end{remark*}

\begin{theorem} \label{mainlast}
For all $n \ge1$ and $q \ge2$, we have
\[
\frac{0.07}{q^{n+1}} \leq\|{\mathcal Q}_{q,n}-{\mathcal
Q}_q\|_{\mathrm{TV}} \leq\frac{2.3}{q^{n+1}}.
\]
\end{theorem}

The following lemma characterizes the two distributions of interest in
this section.

\begin{lemma} 
If $Q$ has the ${\mathcal Q}_q$ distribution, then
\[
E \bigl[qf(Q+1) \bigr]=E \bigl[\bigl(q^{2Q}-1\bigr) f(Q)
\bigr]
\]
for all functions $f$ for which these expectations exist.

If $Q_n$ has the ${\mathcal Q}_{q,n}$ distribution, then
%
\begin{equation}
\label{QnqcharHerm} E \bigl[ \bigl(q - (-1)^{n-Q_n}q^{Q_n-n+1} \bigr)
f(Q_n+1) \bigr]=E \bigl[\bigl(q^{2 Q_n}-1\bigr)
f(Q_n) \bigr]
\end{equation}
for all functions $f$ for which these expectations exist.
\end{lemma}

\begin{pf} For the first assertion, one calculates from (\ref{pkhermitian})
that
\[
\frac{p_{k-1}}{p_k} = \frac{q^{2k}-1}{q}\qquad\mbox{for all $k \in
\mathbb{N}_0$.}
\]
Taking $a(k)=q$ and $b(k)=q^{2k}-1$ in Lemma~\ref{abcharlem}, the
first assertion follows.

For the second assertion, one calculates that
\[
\frac{p_{k-1,n}}{p_{k,n}} = \frac{N(n,n-k+1)}{N(n,n-k)} = \frac
{q^{2k}-1}{q-(-1)^{n-k+1}q^{k-n}}.
\]
Taking $a(k) =q- (-1)^{n-k+1}q^{k-n}$ and $b(k)=q^{2k}-1$ in Lemma~\ref{abcharlem}, and noting $a(n+1)=0$, the second assertion follows.
\end{pf}

Next, we handle the moment $E[q^{Q_n}]$. Unlike all our other moment
computations where we obtain equality, here we derive an upper bound.

\begin{lemma} \label{boundmom} If $Q_n$ has the ${\mathcal Q}_{q,n}$
distribution, then
\[
E\bigl(q^{Q_n}\bigr) \leq2 + q^{-n}.
\]
\end{lemma}

\begin{pf} Setting $f(x)=q^{-x}$ in (\ref{QnqcharHerm}) implies that
\[
E \bigl[ q^{-Q_n} - (-1)^{n-Q_n}q^{-n} \bigr] = E
\bigl[q^{Q_n} - q^{-Q_n}\bigr].
\]
Thus,
\[
E\bigl[q^{Q_n}\bigr] = E \bigl[ 2q^{-Q_n} -
(-1)^{n-Q_n}q^{-n} \bigr] \leq2 + q^{-n}.
\]\upqed
\end{pf}

In the remainder of this section, we consider the Stein equation (\ref
{receq}) for the target distribution ${\mathcal Q}_q$ with
\[
a(k)=q \quad\mbox{and} \quad b(k)=q^{2k}-1
\]
and for $A \subset\mathbb{N}_0$ we let $f_A$ denote the solution
(\ref{steinsol}) when $h(k)={\mathbf1}(k \in A)$.
Our next task is to provide a bound on $f_A$. In the following, we will
apply the identity
%
\begin{equation}
\label{oddeven=oddproduct} \prod_{i\; \mathrm{odd}} \bigl(1-1/q^i
\bigr) \prod_{i\; \mathrm{even}} \bigl(1+1/q^i\bigr) =
\prod_{i\; \mathrm{odd}} \frac{1}{(1+1/q^i)},
\end{equation}
which holds since
\[
\prod_{i\; \mathrm{odd}} \bigl(1-1/q^i\bigr) =
\frac{\prod_i (1-1/q^i)}{\prod_i
(1-1/q^{2i})} = \prod_i \frac{1}{(1+1/q^i)}.
\]

\begin{lemma} \label{normbounds} The function $f_A$ satisfies
\[
\sup_{A \subset\mathbb{N}_0} \bigl|f_A(1)\bigr| \leq\frac{1.1}{q^2}
\quad\mbox{and} \quad\sup_{A \subset\mathbb{N}_0, k \ge2}\bigl|f_A(k)\bigr| \le
\frac
{1.8}{q^4} \qquad\mbox{for all $q \ge2$.}
\]
\end{lemma}

\begin{pf} By Lemma~\ref{fA1boundlem},
\[
\bigl|f_A(1)\bigr| \leq\frac{P(Q \geq1)}{q} = \frac{1-p_0}{q}.
\]
By (\ref{pkhermitian}), (\ref{oddeven=oddproduct}) and the third
claim of Lemma~\ref{bound},
\begin{eqnarray*}
p_0 & = & \prod_{i\; \mathrm{odd}}
\bigl(1-1/q^i\bigr) \prod_{i\; \mathrm{even}}
\bigl(1+1/q^i\bigr)
\\
& \geq& \bigl(1-1/q-1/q^3\bigr) \bigl(1+1/q^2\bigr)
\\
& \geq& 1-1/q-1/q^5.
\end{eqnarray*}
Thus, $1-p_0 \leq1/q + 1/q^5 \leq1.1/q$, and hence $|f_A(1)| \leq
1.1/q^2$, for all $q \ge2$.

For $k \geq1$, by (\ref{prodUUcbound}),
%
\begin{eqnarray}\label{Hermitianfboundsum}
\bigl|f_A(k+1)\bigr| & \leq& \frac{P(Q \geq k+1)}{q p_k}
\nonumber
\\
&=& q^{k^2-1}\prod_{i=1}^k
\bigl(1-1/q^{2i}\bigr)\sum_{l=k+1}^\infty
\frac{1}{q^{l^2}\prod_{j=1}^l (1-1/q^{2j})}
\\
& = & q^{k^2-1} \sum_{l=k+1}^{\infty}
\frac{1}{q^{l^2} \prod_{j=k+1}^l (1-1/q^{2j})}
\nonumber
\\
& \leq& \frac{q^{k^2-1}}{\prod_{j=k+1}^{\infty} (1-1/q^{2j})} \sum
_{l=k+1}^{\infty}
\frac{1}{q^{l^2}}.\nonumber
\end{eqnarray}
Since $k \geq1$, using (\ref{prodagesuma}) we have that
%
\begin{eqnarray}
\label{hermitedefdq} \frac{1}{\prod_{j=k+1}^{\infty} (1-1/q^{2j})} \le
\frac{1}{\prod_{j=2}^{\infty} (1-1/q^{2j})}
\le\frac{1}{1-\sum_{j=2}^\infty q^{-2j}} = d_q
\nonumber\\[-8pt]\\[-8pt]
\eqntext{\displaystyle\mbox{where }
d_q=\frac{1}{1-(1/(q^4-q^2))}.}
\end{eqnarray}
Thus, from (\ref{Hermitianfboundsum}),
%
\begin{eqnarray}\label{Hermitedefsq}
\bigl|f_A(k+1)\bigr| &\le& d_q \cdot q^{k^2-1}\sum
_{l=k+1}^{\infty} \frac
{1}{q^{l^2}}
\nonumber
\\
&=& d_q \cdot q^{k^2-1}\sum_{l=0}^{\infty}
\frac
{1}{q^{(k+1+l)^2}}
\nonumber
\\
&\le& d_q \frac{q^{k^2-1}}{q^{(k+1)^2}}\sum_{l=0}^{\infty}
\frac
{1}{q^{l^2}}
\\
& = & \frac{d_q}{q^{2k+2}}\sum_{l=0}^{\infty}
\frac
{1}{q^{l^2}}
\nonumber
\\
&\le& \frac{d_q s_q}{q^4} \qquad\mbox{where } s_q=\sum
_{l=0}^{\infty} \frac{1}{q^{l^2}},\nonumber
\end{eqnarray}
using $k \ge1$ in the final inequality. Now using that $d_q$ and $s_q$
are decreasing for $q \ge2$, and that
\[
\sum_{l=0}^{\infty} \frac{1}{q^{l^2}} \le1+
\frac{1}{2}+\sum_{l=2}^\infty
\frac{1}{2^{l+2}}=1.625
\]
we obtain the second claim of the lemma.
\end{pf}

Now we present the proof of the main result of this section, Theorem
\ref{mainlast}.

\begin{pf*}{Proof of Theorem~\ref{mainlast}}
We first compute a lower bound for the case where $n$ is
odd. From (\ref{pknhermitian}), we have
\begin{eqnarray*}
p_{0,n} &=& (1-1/q) \bigl(1-1/q^3\bigr) \cdots
\bigl(1-1/q^n\bigr)
\\
&&{}\times \bigl(1+1/q^2\bigr)
\bigl(1+1/q^4\bigr) \cdots\bigl(1+1/q^{n-1}\bigr).
\end{eqnarray*}
By (\ref{pkhermitian}) and (\ref{oddeven=oddproduct}),
\begin{eqnarray*}
p_0 &\geq& (1-1/q) \bigl(1-1/q^3\bigr) \cdots
\bigl(1-1/q^{n+2}\bigr)
\\
&&{}\times
\bigl(1+1/q^2\bigr) \bigl(1+1/q^4\bigr) \cdots
\bigl(1+1/q^{n+1}\bigr)
\\
&=&\bigl(1+1/q^{n+1}\bigr) \bigl(1-1/q^{n+2}\bigr)
p_{0,n}.
\end{eqnarray*}
Thus,
%
\begin{eqnarray}\label{hermitelowerodd}
p_0 - p_{0,n} & \geq& (1-1/q) \bigl(1-1/q^3
\bigr) \cdots\bigl(1-1/q^n\bigr)
\nonumber
\\
& &{} \times\bigl(1+1/q^2\bigr) \bigl(1+1/q^4\bigr)
\cdots\bigl(1+1/q^{n-1}\bigr)
\nonumber
\\
& &{} \times\bigl[ \bigl(1+1/q^{n+1}\bigr) \bigl(1-1/q^{n+2}
\bigr) -1 \bigr]
\nonumber
\\
& \geq& (1-1/q) \bigl(1-1/q^3\bigr) \cdots\bigl(1-1/q^n
\bigr)
\nonumber\\[-8pt]\\[-8pt]
& &{} \times\bigl[ \bigl(1+1/q^{n+1}\bigr) \bigl(1-1/q^{n+2}
\bigr) -1 \bigr]
\nonumber
\\
& \geq& (1-1/q) \bigl(1-1/q^3\bigr) \cdots\bigl(1-1/q^n
\bigr) \biggl[ \frac
{(1-1/q-1/q^3)}{q^{n+1}} \biggr]
\nonumber
\\
& \geq& \bigl(1-1/q-1/q^3\bigr)^2 / q^{n+1}\nonumber
\\
& \geq&0.14 / q^{n+1},
\nonumber
\end{eqnarray}
where the fourth inequality used the third claim of Lemma~\ref{bound}.
Thus, the total variation
distance between ${\mathcal Q}_{q,n}$ and ${\mathcal Q}_q$ is at least
$\frac{1}{2} [ p_0-p_{0,n} ] \ge0.07 /q^{n+1}$.

Now we compute a lower bound for $n$ even. From (\ref{pknhermitian}),
\begin{eqnarray*}
p_{0,n} &=& (1-1/q) \bigl(1-1/q^3\bigr) \cdots
\bigl(1-1/q^{n-1}\bigr)
\\
&&{}\times \bigl(1+1/q^2\bigr)
\bigl(1+1/q^4\bigr) \cdots\bigl(1+1/q^{n}\bigr)
\end{eqnarray*}
and by (\ref{pkhermitian}) and (\ref{oddeven=oddproduct}),
\begin{eqnarray*}
p_0 &\leq& (1-1/q) \bigl(1-1/q^3\bigr) \cdots
\bigl(1-1/q^{n+1}\bigr)
\\
&&{} \times\bigl(1+1/q^2\bigr) \bigl(1+1/q^4\bigr) \cdots
\bigl(1+1/q^{n+2}\bigr)
\\
&=& p_{0,n} \bigl(1-1/q^{n+1}\bigr) \bigl(1+1/q^{n+2}
\bigr).
\end{eqnarray*}
Thus,
%
\begin{eqnarray} \label{hermitelowereven}
p_{0,n} - p_0 & \geq& (1-1/q) \bigl(1-1/q^3
\bigr) \cdots\bigl(1-1/q^{n-1}\bigr)
\nonumber
\\
& &{} \times\bigl(1+1/q^2\bigr) \bigl(1+1/q^4\bigr)
\cdots\bigl(1+1/q^{n}\bigr)
\nonumber
\\
& &{} \times\bigl[1-\bigl(1-1/q^{n+1}\bigr) \bigl(1+1/q^{n+2}
\bigr)\bigr]
\nonumber
\\
& \geq& (1-1/q) \bigl(1-1/q^3\bigr) \cdots\bigl(1-1/q^{n-1}
\bigr)\nonumber
\\
& &{} \times\bigl[1-\bigl(1-1/q^{n+1}\bigr) \bigl(1+1/q^{n+2}
\bigr)\bigr]
\\
& \geq& \frac{(1-1/q)}{q^{n+1}} \prod_{i\; \mathrm{odd}}
\bigl(1-1/q^i\bigr)
\nonumber
\\
& \geq& \frac{(1-1/q)(1-1/q-1/q^3)}{q^{n+1}}\nonumber
\\
& \geq&0.18/ q^{n+1},
\nonumber
\end{eqnarray}
where the fourth inequality used the third claim of Lemma~\ref{bound}.
Thus, the total variation distance\vspace*{1pt} between ${\mathcal Q}_{q,n}$ and
${\mathcal Q}_q$ is at least $\frac{1}{2} [ p_{0,n}-p_0 ] \ge0.09 /q^{n+1}$.

For the upper bound, arguing as in the proof of Theorem~\ref{symmetricnonzerodiagthm},
%
\begin{eqnarray}\label{hermitefinalbound}
&& \bigl|P(Q_n \in A) - P(Q \in A)\bigr|
\nonumber
\\
&&\qquad  =  \bigl|E\bigl[h_A(Q_n)\bigr] - Q_q
h_A \bigr|
\nonumber
\\
&&\qquad =  \bigl|E\bigl[q f_A(Q_n+1) - \bigl(q^{2Q_n}-1
\bigr) f_A(Q_n)\bigr]\bigr|
\nonumber
\\
&&\qquad =  \bigl|E\bigl[(-1)^{n-Q_n} q^{-n+Q_n+1} f_A(Q_n+1)
\bigr] \bigr|
\nonumber
\\
&&\qquad \le q^{-n+1} \bigl|f_A(1)\bigr|P(Q_n=0)
\nonumber\\[-8pt]\\[-8pt]
&&\quad\qquad{}+ E
\bigl[q^{-n+Q_n+1} \bigl|f_A(Q_n+1)\bigr|{\mathbf
1}(Q_n \ge1)\bigr]
\nonumber
\\
&&\qquad \le q^{-n+1} \bigl|f_A(1)\bigr|+ q^{-n+1}E
\bigl[q^{Q_n}\bigr] \sup_{k \ge2} \bigl|f_A(k)\bigr|
\nonumber
\\
&&\qquad \le q^{-n+1} \frac{1.1}{q^2}+ q^{-n+1} \bigl( 2+
q^{-n} \bigr) \frac{1.8}{q^4}
\nonumber
\\
&&\qquad \le q^{-(n+1)} \bigl(1.1+3.6q^{-2} +1.8q^{-3}
\bigr)\nonumber
\\
&&\qquad \le 2.3/q^{n+1}
\nonumber
\end{eqnarray}
for $n \ge1$. The third inequality used Lemmas~\ref{boundmom} and
\ref{normbounds}.
\end{pf*}

\begin{remark}
The distribution $p_{k,n}$ of (\ref{pknhermitian}) holds for $q \ge
3$. Over this range, the bounds of Theorem~\ref{mainlast} may be
slightly improved by applying (\ref{hermitedefdq}) and (\ref
{Hermitedefsq}) to replace 1.8 in Lemma~\ref{normbounds} by 1.4, and
then using this value in (\ref{hermitefinalbound}). One may
similarly improve the lower bound by replacing 0.14 by 0.38 in (\ref
{hermitelowerodd}), and 0.18 by~0.41 in (\ref{hermitelowereven}),
resulting in
\[
\frac{0.19}{q^{n+1}} \leq\|{\mathcal Q}_{q,n}-{\mathcal
Q}_q\|_{\mathrm{TV}} \leq\frac{1.5}{q^{n+1}} \qquad\mbox{for all
$q \ge3$.}
\]
\end{remark}

\setcounter{prop}{0}
\begin{appendix}\label{app}
\section*{Appendix}
The main purpose of this appendix is to give an algebraic proof of
Lemma~\ref{qtoX} in the special case that $m=0$. The proof assumes
familiarity with rational canonical forms of matrices (i.e., the
theory of Jordan forms over finite fields), and with cycle index
generating functions. Background on these topics can be found in \cite
{F1} or \cite{St}, or in the survey \cite{F2}.

\begin{pf*}{Proof of Lemma~\ref{qtoX} when $m=0$}
The sought equation is
%
\begin{equation}
\label{equ} \sum_{k=0}^n
q^k p_{k,n} = 2 - 1/q^n.
\end{equation}
From the expression for $p_{k,n}$ in (\ref{defQqnrect}) specialized
to the case $m=0$, it is clear that if one multiplies (\ref{equ}) by
$q^t (1-1/q) \cdots(1-1/q^n)$
where $t$ is sufficiently large as a function of $n$, then both sides
become polynomials in $q$. Since polynomials in $q$ agreeing
for infinitely many values of $q$ are equal, it is enough to prove the
result for infinitely many values of $q$, so we demonstrate it for $q$
a prime power.

Let $\operatorname{Mat}(n,q)$ be the collection of all $n \times n$ matrices with
entries in $\mathbb{F}_q$ and $M \in \operatorname{Mat}(n,q)$. Then $n$ minus the
rank of $M$ is equal to
$l(\lambda_z(M))$, the number of parts in the partition corresponding
to the degree one polynomial $z$ in the rational canonical
form of $M$.
%
\begin{equation}
\label{exp} E\bigl(q^{Q_n}\bigr) = \frac{1}{q^{n^2}} \sum
_{M
\in \operatorname{Mat}(n,q)} q^{l(\lambda_z(M))},
\end{equation}
where $\operatorname{Mat}(n,q)$ denotes the set of $n \times n$ matrices over the
finite field $\mathbb{F}_q$.

From the cycle index for $\operatorname{Mat}(n,q)$ (Lemma 1 of \cite{St}), it follows that
%
\begin{eqnarray}\label{two}
&& 1 + \sum_{n \geq1} \frac{u^n}{|\operatorname{GL}(n,q)|}
\sum_{M \in \operatorname{Mat}(n,q)} q^{l(\lambda_z(M))}
\nonumber\\[-8pt]\\[-8pt]
&&\qquad = \biggl[ \sum
_{\lambda} \frac{q^{l(\lambda)} u^{|\lambda|}
}{c_{\mathrm{GL},z} (\lambda)} \biggr] \prod
_{\phi\neq z} \sum_{\lambda}
\frac{u^{|\lambda| \deg(\phi
)}}{c_{\mathrm{GL},\phi} (\lambda)}.\nonumber
\end{eqnarray}
Here, $\lambda$ ranges over all partitions of all natural numbers, and
$l(\lambda)$
is the number of parts of $\lambda$. The quantity $c_{\mathrm{GL},\phi
}(\lambda)$ is a certain function of
$\lambda, \phi$ which depends on the polynomial $\phi$ only through
its degree. The product is over all monic, irreducible
polynomials $\phi$ over $\mathbb{F}_q$ other than $\phi= z$.

From the cycle index for $\operatorname{GL}(n,q)$ (Lemma 1 of \cite{St}), it follows that
%
\begin{equation}
\label{three} \frac{1}{1-u} = 1 + \sum_{n \geq1}
\frac{u^n}{|\operatorname{GL}(n,q)|} \sum_{\alpha\in \operatorname{GL}(n,q)} 1 = \prod
_{\phi\neq z} \sum_{\lambda}
\frac{u^{|\lambda| \deg(\phi
)}}{c_{\mathrm{GL},\phi} (\lambda)}.
\end{equation}

Summarizing, it follows from (\ref{two}) and (\ref{three}) that
%
\begin{equation}
\label{four} 1 + \sum_{n \geq1} \frac
{u^n}{|\operatorname{GL}(n,q)|}
\sum_{M \in \operatorname{Mat}(n,q)} q^{l(\lambda_z(M))} = \frac{1}{1-u}
\sum_{\lambda} \frac{q^{l(\lambda)} u^{|\lambda
|}}{c_{\mathrm{GL},z}(\lambda)}.
\end{equation}

The next step is to compute
\[
\sum_{\lambda} \frac{q^{l(\lambda)} u^{|\lambda
|}}{c_{\mathrm{GL},z}(\lambda)} = \sum
_{\lambda} \frac{q^{l(\lambda)}
u^{|\lambda|}}{c_{\mathrm{GL},z-1}(\lambda)}.
\]
This equality holds because $c_{\mathrm{GL},\phi}(\lambda)$ depends on the
polynomial $\phi$ only
through its degree.
From the cycle index of $\operatorname{GL}(n,q)$, it follows that
\begin{eqnarray*}
& & 1 + \sum_{n \geq1} \frac{u^n}{|\operatorname{GL}(n,q)|} \sum
_{\alpha\in
\operatorname{GL}(n,q)} q^{l(\lambda_{z-1}(\alpha))}
\\
&&\qquad =  \sum_{\lambda} \frac{q^{l(\lambda)} u^{|\lambda
|}}{c_{\mathrm{GL},z-1}(\lambda)} \prod
_{\phi\neq z,z-1} \sum_{\lambda}
\frac{u^{|\lambda| \deg(\phi)}}{c_{\mathrm{GL},\phi}(\lambda)}
\\
&&\qquad =  \frac{ \sum_{\lambda} (({q^{l(\lambda)} u^{|\lambda
|}})/({c_{\mathrm{GL},z-1}(\lambda)})) } {
\sum_{\lambda} ({u^{|\lambda| }}/({c_{\mathrm{GL},z-1}(\lambda)})) } \prod_{\phi
\neq z} \sum_{\lambda} \frac{u^{|\lambda| \deg(\phi
)}}{c_{\mathrm{GL},\phi}(\lambda)}
\\
&&\qquad =  \frac{1}{1-u} \frac{ \sum_{\lambda} (({q^{l(\lambda)}
u^{|\lambda|}})/({c_{\mathrm{GL},z-1}(\lambda)})) } {
\sum_{\lambda} ({u^{|\lambda|}}/({c_{\mathrm{GL},z-1}(\lambda)})) }
\\
&&\qquad =  \frac{\prod_{i \geq1} (1-u/q^i)}{1-u} \sum_{\lambda}
\frac
{q^{l(\lambda)} u^{|\lambda|}}{c_{\mathrm{GL},z-1}(\lambda)}.
\end{eqnarray*}
The third equality used (\ref{three}) and the final equality is from
Lemma 6 of \cite{St}
and page 19 of \cite{A}.

Next, we can use group theory to find an alternate expression for
\[
1 + \sum_{n \geq1} \frac{u^n}{|\operatorname{GL}(n,q)|} \sum
_{\alpha\in \operatorname{GL}(n,q)} q^{l(\lambda_{z-1}(\alpha))}.
\]
Indeed, by the theory of rational canonical forms, $q^{l(\lambda
_{z-1}(\alpha))}$ is the number of fixed points of $\alpha$ in its
action on the underlying $n$ dimensional vector space $V$.
By Burnside's lemma (page 95 of \cite{VW}), the average number of
fixed points of a finite group acting on a finite set is the number of
orbits of the action on the set. For $\operatorname{GL}(n,q)$ acting on $V$, there are
two such orbits, consisting of the zero vector and the set of nonzero
vectors. Thus,
\[
1 + \sum_{n \geq1} \frac{u^n}{|\operatorname{GL}(n,q)|} \sum
_{\alpha\in \operatorname{GL}(n,q)} q^{l(\lambda_{z-1}(\alpha))} = 1 + \sum
_{n \geq1} 2u^n = \frac{1+u}{1-u}.
\]

Comparing the final equations of the previous two paragraphs gives that
%
\begin{equation}
\label{six} \sum_{\lambda} \frac{q^{l(\lambda)} u^{|\lambda
|}}{c_{\mathrm{GL},z}(\lambda)} =
\frac{1+u}{\prod_{i \geq1} (1-u/q^i)}.
\end{equation}
It follows from (\ref{four}) and (\ref{six}) that
\[
1 + \sum_{n \geq1} \frac{u^n}{|\operatorname{GL}(n,q)|} \sum
_{M \in \operatorname{Mat}(n,q)} q^{l(\lambda_z(M))} = \frac{1+u}{1-u} \prod
_{i \geq1} \frac
{1}{1-u/q^i}.
\]
Thus, by (\ref{exp}), $E(q^{Q_n})$ is $\frac{|\operatorname{GL}(n,q)|}{q^{n^2}}$
multiplied by the coefficient of $u^n$ in
\[
\frac{1+u}{1-u} \prod_{i \geq1}
\frac{1}{1-u/q^i}.
\]
From page 19 of \cite{A}, the coefficient of $u^n$ in
\[
\frac{1}{1-u} \prod_{i \geq1} \frac{1}{1-u/q^i}
\]
is equal to $[(1-1/q)(1-1/q^2) \cdots(1-1/q^n)]^{-1}$. Thus,
\begin{eqnarray*}
E\bigl(q^{Q_n}\bigr)
&=&  \frac{|\operatorname{GL}(n,q)|}{q^{n^2}} \biggl[ \frac{1}{(1-1/q) \cdots
(1-1/q^n)} + \frac{1}{(1-1/q) \cdots(1-1/q^{n-1})}
\biggr]
\\
&=&  2 - \frac{1}{q^n},
\end{eqnarray*}
where the last equality used that $|\operatorname{GL}(n,q)| = q^{n^2} (1-1/q) \cdots
(1-1/q^n)$.
\end{pf*}

We close this section with two remarks about the distribution
${\mathcal Q}_{q,n}$ in (\ref{defQqnrect}) (for general $m$) from
the \hyperref[intro]{Introduction}.

\begin{itemize}
\item From \cite{Be}, there is a natural Markov chain on $\{0,1,\ldots,n\}$ which has ${\mathcal Q}_{q,n}$ as its stationary distribution.
This chain
has transition probabilities
\begin{eqnarray*}
M(i,i+1) &=& \frac{q^{n-i-1} (q^{n-i} - 1)}{(q^n-1) (q^{n+m}-1)},
\\
M(i,i-1) &=& \frac{(q^n-q^{n-i})(q^{n+m}-q^{n-i})}{(q^n-1)(q^{n+m}-1)},
\\
M(i,i) &=& 1 - M(i,i-1) - M(i,i+1).
\end{eqnarray*}
This Markov chain describes how the rank of a matrix evolves by adding
a uniformly chosen rank one matrix at each step.

\item The following known lemma gives a formula for the chance that a
random $k \times n$ matrix with entries from $\mathbb{F}_q$ has rank
$r$. For its statement, we let
\[
{\lleft[\matrix{n
\cr
m} \rright]_q} = \frac{(q^n-1)(q^{n-1}-1)
\cdots(q^{n-m+1}-1)}{(q^m-1)
(q^{m-1}-1) \cdots
(q-1)}
\]
be the $q$-binomial coefficient.
\end{itemize}

\begin{lemma}[(\cite{VW}, page 338)]\label{first}
The chance that a
random $k \times n$ matrix with entries from $\mathbb{F}_q$ has rank
$r$ is equal to
%
\begin{equation}
\label{one} \frac{1}{q^{kn}} {\lleft[\matrix{n
\cr
r}
\rright]_q} \sum_{l=0}^r
(-1)^{r-l} {\lleft[\matrix{r
\cr
l} \rright]_q}
q^{kl+
{r-l \choose2}}.
\end{equation}
\end{lemma}

Following a suggestion of Dennis Stanton, we indicate how Lemma~\ref{first} can be used to derive the product formula for $p_{k,n}$ in the
\hyperref[intro]{Introduction}. By replacing $k$ by $n$, $n$ by $n+m$, and $r$ by $n-k$
in (\ref{one}), we get that the probability that a random $n \times
(n+m)$ matrix has rank $n-k$ is equal to
\begin{eqnarray*}
& & \frac{1}{q^{n(n+m)}} {\lleft[\matrix{n+m
\cr n-k} \rright]_q}
\sum_{l=0}^{n-k} (-1)^{n-k-l} {
\lleft[\matrix{n-k
\cr
l} \rright]_q} q^{nl + {n-k-l \choose2}}
\\
&&\qquad  =  \frac{1}{q^{n(n+m)}} {\lleft[\matrix{n+m
\cr
n-k} \rright]_q}
\sum_{l=0}^{n-k} (-1)^l {
\lleft[\matrix{n-k
\cr
l} \rright]_q} q^{n(n-k-l) + {l \choose2}}
\\
&&\qquad =  \frac{1}{q^{n(m+k)}} {\lleft[\matrix{n+m
\cr
n-k} \rright]_q}
\sum_{l=0}^{n-k} \biggl[ \frac{-1}{q^{n+1}}
\biggr]^l {\lleft[\matrix{n-k
\cr
l} \rright]_q}
q^{{l+1 \choose2}}.
\end{eqnarray*}
Plugging into the $q$-binomial theorem (page 78 of \cite{Br})
\[
(1+xq) \bigl(1+xq^2\bigr) \cdots\bigl(1+xq^r\bigr) =
\sum_{l=0}^r {\lleft[\matrix{r
\cr
l}
\rright]_q} q^{l(l+1)/2} x^l
\]
with $r=n-k$ and $x = -1/q^{n+1}$ gives that the probability that a
random $n \times(n+m)$ matrix over $\mathbb{F}_q$ has rank $n-k$ is
equal to
\[
\frac{1}{q^{n(m+k)}} {\lleft[\matrix{n+m
\cr
n-k} \rright]_q}
\bigl(1-1/q^n\bigr) \cdots\bigl(1-1/q^{k+1}\bigr).
\]
It follows from elementary manipulations that this is equal to
\[
\frac{1}{q^{k(m+k)}} \frac{\prod_{i=1}^{n+m} (1-1/q^i) \prod_{i=k+1}^n
(1-1/q^i)}{\prod_{i=1}^{n-k} (1-1/q^i) \prod_{i=1}^{m+k}
(1-1/q^i)}.
\]
\end{appendix}

\section*{Acknowledgements}
The authors thank Dennis Stanton and the referees for helpful comments.




\printaddresses

\end{document}